\documentclass[12pt]{article}
\usepackage{geometry}
\usepackage{graphicx}
\usepackage{latexsym}
\usepackage{amsthm}
\usepackage{amsmath}
\usepackage{amssymb}
\usepackage[utf8]{inputenc}
\usepackage{stmaryrd}
\usepackage{bm}
\usepackage{bbm}
\usepackage[square,authoryear]{natbib}
\usepackage{float}
\usepackage[lofdepth,lotdepth]{subfig}
\usepackage{appendix}
\usepackage{algorithm}
\usepackage{algpseudocode}
\usepackage{algorithmicx}
\usepackage{authblk}
 
\usepackage{xcolor}
 
\newtheorem{theorem}{Theorem}
\newtheorem{lemma}{Lemma}
\newtheorem{proposition}{Proposition}
\newtheorem{corollary}{Corollary}
\newtheorem{definition}{Definition}
\newcommand{\intd}{\mathrm{d}}
\newcommand{\U}{\mathcal{U}}
\newcommand{\R}{\mathbb{R}}
\newcommand{\N}{\mathbb{N}}
\newcommand{\I}{\mathbb{I}}
\newcommand{\expect}{\mathbb{E}}
\newcommand{\proba}{\mathbb{P}}
\newcommand{\F}{\mathcal{F}}
\newcommand{\probaspace}{\mathcal{P}}
\newcommand{\domainD}{\mathcal{D}}
\newcommand{\X}{\mathcal{X}}
\newcommand{\Z}{\mathcal{Z}}
\newcommand{\C}{\mathcal{C}}
\newcommand{\T}{\textsf{T}}
\newcommand{\G}{\mathcal{G}}
\newcommand{\Div}{\mathrm{div}}
\newcommand{\Y}{\mathcal{Y}}
\newcommand{\setS}{\mathcal{S}}

\newcommand{\conv}[3]{\underset{#1\rightarrow #2}{\overset{#3}{\longrightarrow}}}

\newcommand{\Cov}{\mathrm{Cov}}

\title{Mean field approximation of a heterogeneous population of plants in competition}
\author[,1]{Antonin Della Noce\thanks{Corresponding author : \texttt{antonin.della-noce@centralesupelec.fr}}}
\author[2]{Amélie Mathieu}
\author[1]{Paul-Henry Cournède}
\affil[1]{\normalsize Laboratoire MICS, CentraleSupélec, Université Paris-Saclay, 91190, Gif sur Yvette, France}
\affil[2]{\normalsize INRA AgroPariTech, Route de la Ferme, 78850, Thiverval-Grignon, France}
\date{}

\begin{document}
\maketitle

\begin{abstract}
The processes of interplant competition within a field are still poorly understood. However, they explain a large part of the heterogeneity in a field and may have longer-term consequences, especially in mixed stands. Modeling can help to better understand these phenomena but requires simulating the interactions between different individuals. In the case of large populations, assessing the parameters of a heterogeneous population model from experimental data is intractable computationally. This paper investigates the mean-field approximation of large dynamical systems with random initial conditions and individual parameters, and with interaction being represented by pairwise potentials between individuals. Under this approximation, each individual is in interaction with an infinitely-crowded population, summarized by a probability measure, the mean-field limit distribution, being itself the weak solution of a non-linear hyperbolic partial differential equation. In particular, the phenomenon of chaos propagation implies that the individuals are independent asymptotically when the size of the population tends towards infinity. This result provides perspectives for a possible simplification of the inference problem. The simulation of the mean-field distribution, consisting in a semi-Lagrangian scheme with an interpolation step using Gaussian process regression, is illustrated for a heterogeneous population model representing plants in competition for light.
\end{abstract}

\section{Introduction}
\label{intro}
The interest for modelling heterogeneous populations of plants is on the rise, especially due to the development of the practice of mixed cropping (\cite{Malezieux2009}). Mixing different varieties or different species (\cite{Tang2018}) may have various advantages, such as nitrogen transfer from one species to another, resistance of the population to disease and pests (\cite{Gurr2003}), or enhanced production quality (\cite{Gooding2007} for wheat). However, up to our knowledge, very few models are made to understand the emerging properties of such mixture and to design optimal crops (cf. \cite{Gaudio2019} for a review). A convenient framework for modelling heterogeneous populations is hierarchical modelling, also known as mixed-effects modelling (\cite{Schneider2006}, \cite{Lv2008}, \cite{Baey2016}). A classical formulation of hierarchical model of a population dynamics can be represented as a dynamical system, whose initial conditions and parameters are independent and identically distributed random variables.
\begin{equation}
  \forall i \in \llbracket 1;N\rrbracket,~\left\{\begin{array}{l}
  (  X_i^0, \theta_i) \sim \mu_0\mathrm{~a~probability~measure}\\
   X_i(0) = X_i^0\\
  \displaystyle \frac{\intd  X_i(t)}{\intd t} = F\left( X_i(t),\theta_i;( X_j(t), \theta_j)_{1\leq j\leq N}\right)
  \end{array}\right.
  \label{system_initial}
\end{equation}
$N$ is the number of individuals in the population. Each individual is indexed by integer $i\in \llbracket 1;N\rrbracket$ and is described by a state variable $ X_i$ and individual parameter $ \theta_i$. The state variable $ X_i$ represents time-varying features of the plant, e.g. the size of its aerial part, the total leaf area, etc. The individual parameter $ \theta_i$ represents intrinsic characteristics of individual $i$, that are assumed to be constant throughout the considered time period, and that have influence on the population dynamics. $F$ is a function modelling the influence of the whole population, consisting in the collection $( X_j, \theta_j)_{1\leq j\leq N}$, on the individual development of each plant. A specific form of $F$ is going to be studied in the present article (see equation (\ref{system_g})). The heterogeneity of the population is represented by the probability measure $\mu_0$, that distributes the initial state variable and individual parameter to each individual at the population level. If the marginal distribution of variable $ \theta$, $\mu_0^\theta$, is not reduced to a Dirac distribution, or equivalently if $ \theta$ is not constant over the population, then the population is said to be heterogeneous, as it gathers individuals with different characteristics. The case of homogeneous population has been investigated for example in \cite{Cournede2007}, \cite{Sievanen2008} focusing on the competition between plants.

The problem of statistical inference on such population model consists in identifying distribution $\mu_0$ and function $F$ from collected observation data. In the case where the plants do not interact with each other, various forms of Expectation - Maximization (EM) algorithm, introduced by \cite{Dempster1977}, can be applied to estimate the parameters (\cite{Baey2016}),\cite{Baey2018}, or direct Bayesian inference (\cite{Viaud2018}, chapter 4). Most common forms of EM algorithm and of direct Bayesian inference require a random exploration of the unknown parameter space using Metropolis-Hasting (MH) algorithm, or Metropolis-Hasting within Gibbs (MHWG) algorithmn, which are not suited for the exploration of high-dimensional space in terms of convergence time (\cite{Katafygiotis2008}). Nevertheless, EM algorithm or MHWG algorithm remain efficient tools for parameter estimation in a population model without interaction.

The relative effectiveness of these algorithms is challenged when taking into account interactions within the population model. The correlations between individuals hinder the distribution of the computation and the search space where MH algorithm is applied is of too high dimension, proportional to the number $N$ of individuals (cf. the computational issue encountered in \cite{Schneider2006}). The aim of this research is to suggest other methods more suited to this problem.

A possible research direction is given by variational Bayesian approximation (cf. \cite{Bishop2006}, chapter 10, for an introduction). This method consists in projecting the joint distribution of the random variables $( X_i(t),\theta_i)_{1\leq i\leq N}$, which is non-factorized due to individuals interaction, onto a tensor product of parametric distributions. For specific expression of the function $F$, such as the interaction function used in the Cucker-Smale model (\cite{Cucker2007}, \cite{Carrillo2010}), any subset of the the population has a joint distribution asymptotically factorized as $N\rightarrow +\infty$, a phenomenon referred as chaos propagation in the literature (\cite{Bolley2011}). Qualitatively, when the population is infinite, the states at time $t$ and parameters of the individuals behave as if they were independent random variables distributed according to a single probability measure $\mu[t]$, that is called the mean field limit (MFL) distribution.

The question on how to integrate the MFL distribution $\mu[t]$ into the process of statistical inference is beyond the scope of this article. The first step is to check for which kind of heterogeneous population models we can obtain theoretical existence and uniqueness of MFL distribution, along with asymptotic factorization property. We have considered plant population models for which the interaction function $F$ can be decomposed as a sum of elementary interaction functions over the whole population.
\begin{equation}
   \forall i \in \llbracket 1;N\rrbracket,~F\left( X_i(t),\theta_i;( X_j(t), \theta_j)_{1\leq j\leq N}\right) = \frac{1}{N-1}\sum_{1\leq j\leq N, j\neq i} g( X_i(t), \theta_i, X_j(t), \theta_j)
   \label{system_g}
\end{equation}
Such formulation is used in \cite{Schneider2006}, \cite{Lv2008}, \cite{Nakagawa2015}. These models focus on the competition for light within the plant population. \cite{Schneider2006} suggests various models coupling plant development and competition. Amongst the models being smooth enough, we have chosen the one with the most statistical relevance. This model is described in section 2. Equation (\ref{system_g}) is quite close in its formulation to particle systems studied in kinetic equation theory (\cite{Carrillo2010}). The normalization by the size of the population is important for the study of the asymptotic behavior of the population as $N$ tends towards infinity. The derivative of an individual state remains of the same order of magnitude when the size of the population changes. In our case, this normalization is part of the model expression, but for other systems (such as the ones studied in statistical physics), this normalization can be interpreted as a change of time scale (\cite{Golse2013}). Other normalization can be considered in some flocking model, like Vicsek model (\cite{Vicsek1995}, \cite{Degond2018}) where the velocity is normalized by the sum of all the velocities in the population. We shall specify in the next section the assumptions to be made on $g$ and on $\mu_0$ to derive the MFL distribution. We give also an example of plant competition model from \cite{Schneider2006} to illustrate the theoretical development in section 3 to prove existence and uniqueness of MFL distribution, and finally chaos propagation. As our initial aim is to be able to use the MFL distribution for statistical inference, we present a preliminary work in section 4 to approximate this distribution.

\section{Example and assumptions}

\subsection{Working assumptions and notations}
\label{sub_working_assumptions}

In this subsection, we specify our assumptions on the systems (\ref{system_initial}). Let $(\Omega,\F,\proba)$ be a probability space. Let $\X$ be an Euclidean space of dimension $d_\X$ and $\Theta$ be a compact subset of an Euclidean space, such that the dimension of $\Theta$ is $d_\Theta$. The phase space is denoted by $\Z = \X\times \Theta$ and the set of probability measures defined over $\Z$ is denoted by $\probaspace(\Z)$. The set of probability measures $\probaspace(\Z)$ is associated to the space of random variables, i.e. the space of functions $f:\Omega\rightarrow \Z$ measurable for the measure $\proba$. $\Z$ is often endowed with the Lebesgue measure, denoted by $\lambda^{\otimes d_z}$ ($d_z = \mathrm{dim}(\Z)$). Unless otherwise stated, the metric used on $\Z$ is defined by
\begin{equation}
  \forall  z = ( X, \theta) \in \Z,~| z| = \sum_{i = 1}^{d_\X+d_\Theta}\frac{|z_i|}{|z_i^*|}
\end{equation}
where $ z^*\in \Z$ is a reference vector with components all non-zero, and $|.|:a\in \R\mapsto |a|$ is the absolute norm over $\R$. The norm $| z|$ is therefore a dimensionless quantity. Similarly, we use the notation $\displaystyle \forall X\in \X,~|X| = \sum_{i = 1}^{d_\X}\frac{|X_i|}{|X_i^*|}$ and $\displaystyle\forall \theta\in \Theta,~ |\theta| = \sum_{i = 1}^{d_\Theta}\frac{|\theta_i|}{|\theta_i^*|}$. We consider the population model of initial distribution $\mu_0\in \probaspace(\Z)$ a probability measure over $\Z$ and of interaction function $g:\Z\rightarrow \X$.
\begin{equation}
  \begin{aligned}
    &(X_i^0,\theta_i)_{1\leq i\leq N}\sim \mu_0^{\otimes N}\\
    &\forall i\in \llbracket 1;N\rrbracket,~\left\{\begin{array}{l}
    X_i(0) = X_i^0\\
    \forall t\in \R_+,~\displaystyle \frac{\intd X_i(t)}{\intd t} = \frac{1}{N-1}\sum_{j\neq i} g( X_i(t), \theta_i,  X_j(t), \theta_j)
    \end{array}\right.
  \end{aligned}
  \label{system_micro}
\end{equation}
We shall use two notations for the interaction function $g$ : either $g:({X_1,\theta_1,X_2,\theta_2})\in \Z^2\mapsto g({X_1,\theta_1,X_2,\theta_2})\in \X$, either $g:({z_1,z_2})\in \Z^2\mapsto g({z_1,z_2})\in \X$. Here are some assumptions on the smoothness of function $g$.
\paragraph{(A1) Assumption 1 : } There exists $K_1 >0$ such that for all $ X_1, X_2\in \X$ and all $ \theta_1, \theta_2\in \Theta$ $|g({X_1,\theta_1,X_2,\theta_2})|\leq K_1(1+|X_1|+|X_2|)$.
This assumption makes possible the existence of global solution over $\R_+$.

\paragraph{(A2) Assumption 2 : }There exists $K_2 > 0$ such that for all $X_1,X_1',X_2,X_2'\in \X$ and $\theta,\theta'\in \Theta$ we have
\begin{equation}
  |g(X_1,\theta,X_1',\theta')-g(X_2,\theta,X_2',\theta')|\leq K_2(1+|X_1'|+|X_2'|)(|X_1-X_2| + |X_1'-X_2'|)
\end{equation}

\newcommand{\M}{\mathcal{M}}
\paragraph{(A3) Assumption 3} : The transition function $g$ has a partial derivative with respect to the variable $X$, $(X,\theta,X',\theta')\in \Z^2\mapsto \displaystyle \frac{\partial g}{\partial X}(X,\theta,X',\theta')\in  \M_{d_\X}(\R)$, which is continuous and which is such that there exists $K_3 > 0$
\begin{equation}
  \begin{aligned}
    &\forall (X_1,\theta_1,X_2,\theta_2)\in \Z^2,\\
    &\left|\frac{\partial g}{\partial X}(X_1,\theta_1,X_1,\theta_2)\right| = \sup_{X\in \X,|X| = 1}\left|\frac{\partial g}{\partial X}(X_1,\theta_1,X_1,\theta_2).X\right|\leq K_3(1+|X_2|)
  \end{aligned}
\end{equation}

\paragraph{(A4) Assumption 4} : There exists a constant $K_4> 0$ such that for all $X,X'\in \X$ and $\theta_1,\theta_1',\theta_2,\theta'_2\in \Theta$\footnote{As $\Theta$ is not a vector space, we can have $\theta_1-\theta_2$ not belonging to $\Theta$. The notation $|\theta_1-\theta_2|$ has therefore to be understood as the norm of the vector $\theta_1-\theta_2$ in the Euclidean space containing the compact subset $\Theta$.}
\begin{equation}
  |g(X,\theta_1,X',\theta_1')-g(X,\theta_2,X',\theta'_2)|\leq K_4(1+|X|+|X'|)(|\theta_1-\theta_2|+|\theta'_1-\theta'_2|)
\end{equation}
The next subsection gives an example of differential system, where the interaction function $g$ satisfies all four assumptions listed above.

\subsection{Example of Schneider model}
\label{sub_schneider}

The article of \cite{Schneider2006} studies a population of plants (\textit{Arabidopsis thaliana}) in competition for light resources. A dozen of models, more or less empirical, are suggested in this paper to represent the growth of the aerial part of plants subject to the shade of its surroundings, and all these models are compared statistically against experimental data. A similar approach is carried out in \cite{Nakagawa2015} at the scale of a whole forest, observed for several decades. The population was then assumed to be homogeneous, certainly because of the computational issues previously mentioned.

In this model, the soil and water resources are assumed to be in abundance, so that the competition concerned only the light resource. Therefore, only the aerial part\footnote{In the case of \textit{A. thaliana}, $s$ can be the diameter of the rosette (see figure 1)} of the plant is represented by the model. A plant is described by the size of its aerial part $s$, its position $\vec{x} = (x,y)$ in the plane, and by two intrinsic factors $\gamma$ and $S$, determining properties of the individual growth. Over the time, only plants' sizes change. The assumptions of the model are the following :
\begin{enumerate}
\item If the plant grows in isolation, or if the influence of competitors can be neglected, the dynamics of its growth is given by a Gompertz function (\cite{Paine2012}).
  \begin{equation}
    \left\{\begin{array}{l}
    s(0) = s^0\\
    \displaystyle \forall t\in \R_+,~\frac{\intd s(t)}{\intd t} = \gamma s(t)\log\left(\frac{S}{s(t)}\right)
    \end{array}\right.\Rightarrow s(t) = S\exp\left(-e^{-\gamma t}\log\left(\frac{S}{s^0}\right)\right)
  \end{equation}
  The size of the plant converges towards an equilibrium size $S$ with rate $\gamma$. In a more accurate modelling, this equilibrium size should be a function of the environmental conditions, but they are not taken into account here (the light environment is assumed to be controlled). The initial size of the plant $s^0 > 0$ can be thought as the size of the sprout just after emergence.
\item If the plant grows in presence of competitors, in a population consisting of $N$ individuals, the equilibrium size $S_i$ of the individual $i\in \llbracket 1;N\rrbracket$ is perturbed by a factor representing the negative impact of the competition.
  \begin{equation}
    \begin{aligned}
      & \forall i\in \llbracket 1;N\rrbracket,\\
      &\left\{\begin{array}{l}
      s_i(0) = s_i^0\\
      \displaystyle \forall t\in \R_+,~\frac{\intd s_i(t)}{\intd t} = \gamma_is_i(t)\left(\log\left(\frac{S_i}{s_m}\right)\left(1-\frac{1}{N-1}\sum_{j\neq i}C(s_i(t),s_j(t),|\vec x_i-\vec x_j|)\right)\right.\\
      \displaystyle \left.-\log\left(\frac{s_i(t)}{s_m}\right)\right)
      \end{array}\right.
    \end{aligned}
    \label{system_schneider}
  \end{equation}
  where $\displaystyle C(s_i,s_j,|\vec x_i-\vec x_j|) = \frac{\log(s_j/s_m)}{2R_M\displaystyle\left(1+\frac{|\vec{x}_i-\vec{x}_j|^2}{\sigma_x^2}\right)}\left(1+\tanh\left(\frac{1}{\sigma_r}\log\left(\frac{s_j}{s_i}\right)\right)\right)$ with $s_m,\sigma_x,\sigma_r$ being known positive constants and  $R_M$ such that $\forall i\in \llbracket 1;N\rrbracket~\displaystyle \log\left(\frac{S_i}{s_m}\right)\leq R_M$.
\end{enumerate}
\begin{figure}[ht]
  \begin{center}
    \includegraphics[width = 0.5\textwidth]{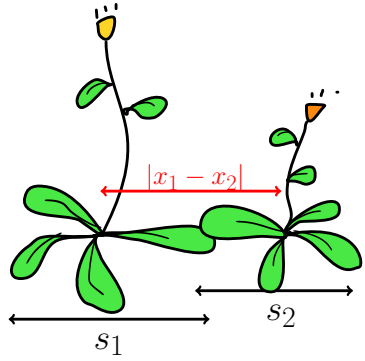}
  \end{center}
  \caption{Parametrization of the competition model : the competition exterted by plant 2 over plant 1 depends on the respective sizes of the plants and on their distance.}
\end{figure}
In presence of competition, the available light environment of plant $i$, represented by the term $\displaystyle \log\left(\frac{S_i}{s_m}\right)$, is reduced by a competition factor $\displaystyle 1-\frac{1}{N-1}\sum_{1\leq j\leq N,j\neq i}C(s_i(t),s_j(t),|\vec x_i-\vec x_j|)$, which is dimensionless and takes values in $[0;1]$. The competition exerted on plant $i$ is all the more important than other plants are
\begin{enumerate}
\item tall in absolute terms, with the factor $\displaystyle\frac{\log(s_j/s_m)}{R_M}$
\item taller than plant $i$, with the factor $\displaystyle \frac{1}{2}\left(1+\tanh\left(\frac{1}{\sigma_r}\log\left(\frac{s_j}{s_i}\right)\right)\right)$
\item close to plant $i$, with the factor $\displaystyle\frac{1}{\displaystyle\left(1+\frac{|\vec{x}_i-\vec{x}_j|^2}{\sigma_x^2}\right)}$
\end{enumerate}

There exist more realistic and complex models to represent competition for light in plants population. \cite{Beyer2015} describes tree crowns development by a transport equation on foliage density. In this model, light ressource is allocated to the different individuals proportionally to their foliage volume. More mechanistic models can be found in the literature, namely the ones making use of Functional Structural Plant Models (FSPM), where the light environment is directly computed by ray tracing through a 3D reconstruction of the canopy (\cite{Cieslak2008}). Such models of competition are still too complex for the method we describe in this article.

Before going any further, we need to prove that system (\ref{system_schneider}) is well-posed for any initial condition.
\begin{proposition}
  Let us consider the initial conditions $(s_i^0)_{1\leq i\leq N}\in (\R_+^*)^N$ and the collection of parameters $(x_i,y_i,S_i,\gamma_i)\in (\R^2\times (\R_+^*)^2)^N$. Then the system (\ref{system_schneider}) has an unique solution $ s_{1:N}:t\mapsto (s_i(t))_{1\leq i\leq N}$ defined over $\R_+$ taking positive values, i.e. verifying $\forall t\in \R_+,~\forall i\in \llbracket 1;N\rrbracket,~s_i(t) >0$.
\end{proposition}

The existence of the solution of this system is a classical application of Cauchy-Lipschitz theorem to the system satisfied by the vector $r_{1:N}(t) = \left(\log\left(\frac{s_i(t)}{s_m}\right)\right)_{1\leq i\leq N}\in \R^N$. Details of the proof can be found in appendix \ref{proof_prop1}. We need also to check for which conditions the global solution given by proposition 1 is consistent with the biological assumptions of the model. A solution $s_{1:N}:t\in \R_+\mapsto s_{1:N}(t)\in \R_+^N$ is consistent with the assumptions of the model if it meets the following constraints :
\begin{itemize}
\item The size of each individual must remain below its equilibrium size and above the minimal size $s_m$, i.e. for all $i\in \llbracket 1;N\rrbracket,~ s_m< s_i(t)\leq S_i$.
\item The competition factor must remain in $[0;1]$, i.e. for all $i,j\in \llbracket 1;N\rrbracket$, $C(s_i(t),s_j(t),| x_i- x_j|)\in [0;1]$.
\end{itemize}
These conditions can be met if we set some conditions on the support of the initial distribution $\mu_0$. The next proposition gives sufficient conditions on the support of $\mu_0$ for these constraints to be verified.

\begin{proposition}
  Let $\displaystyle \mathcal{D} = \{(s,x,y,S,\gamma)\in [s_m;+\infty[\times \R\times \R\times [s_m;+\infty[\times \R_+|s_m < S\leq s_me^{R_M},~s_m<s\leq S\}$. Let $\mu_0$ be a probability over $\R^5$, i.e. $\mu_0\in \probaspace(\R^5)$, such that the support of $\mu_0$ is included in the interior of domain $\domainD$. Let $ Z_N^0 = (s_i^0,x_i,y_i,S_i,\gamma_i)_{1\leq i\leq N}$ be a random variable of distribution $\mu_0^{\otimes N}$ and $t\in \R_+\mapsto s_{1:N}(t, Z_N^0)$ the solution of system (\ref{system_schneider}) with initial configuration $ Z_N^0$. Then we have almost surely that for all time $t\in \R_+$, $(s_i(t, Z_N^0),x_i,y_i,S_i,\gamma_i)_{1\leq i\leq N}\in \mathring{\domainD}^N$, the interior of domain $\domainD^N$.
\end{proposition}
The proof of this proposition can be found in appendix \ref{proof_prop2}. It is based on the fact that within domain $\mathcal{D}$, the evolution of each plant size is bounded between two growth rates, ensuring the size to remain within a biologically consistent interval.
\begin{equation*}
  \forall i\in \llbracket 1;N\rrbracket,~\forall t\in \R_+,~ \gamma_is_i(t)\log\left(\frac{s_m}{s_i(t)}\right)\leq \frac{\intd s_i(t)}{\intd t}\leq \gamma_is_i(t)\log\left(\frac{S_i}{s_i(t)}\right)
\end{equation*}
The trajectories associated respectively to the upper and the lower bound remain within the domain $\mathcal{D}$ when the initial condition is generated by a $\mu_0$ satisfying the assumptions of proposition 2.

For the sake of clarity, we give also an example of initial distribution $\mu_0$, that is the source of heterogeneity and randomness in the system represented by equation (\ref{system_schneider}). Let $(s^0,x,y,S,\gamma)$ be a random variable of distribution $\mu_0$. We have chosen the distribution $\mu_0$ such that the positions of plants are mutually independent, but with a spatial pattern on parameters $\gamma$ and $S$. In what follows, $\U([a;b])$ is the notation for the uniform distribution over the segment $[a;b]$.
\begin{equation}
  \begin{aligned}
    &x\sim \U([0;L])~~{y}\sim\U([0;L])\text{ and }x,y\text{ are independent}\\
    &S|x\sim\U([S_1(x);S_2(x)])\\
    &\text{with }S_1(x) = S_m + \frac{x}{L}(S_M-\sigma_S-S_m),~~S_2(x) = S_m + \sigma_S + \frac{x}{L}(S_M-\sigma_S-S_m)\\
    &\text{with }S_m\text{ and }\sigma_S\text{ such that }S_m > 0,~\sigma_S > 0,~S_m+\sigma_S < S_M\\
    &\gamma|y\sim\U([\gamma_1(y);\gamma_2(y)])\\
    &\text{with }\gamma_1(y) = \gamma_m + \frac{y}{L}(\gamma_M-\sigma_\gamma-\gamma_m),~~\gamma_2(y) = \gamma_m + \sigma_\gamma + \frac{y}{L}(\gamma_M-\sigma_\gamma-\gamma_m)\\
    &s^0\sim\delta_{s^0} \text{(the initial size of the plants is a constant over the population)}\\
  \end{aligned}
  \label{eq_mu0_schneider}
\end{equation}
This initial distribution $\mu_0$ implies that the plants are evenly distributed over the square $[0;L]^2$, that the plants with large values of $x$ are likely to be tall, and the ones with high values of $y$ are likely to grow fast. In this example, the initial distribution $\mu_0$ is not absolutely continuous with respect to the Lebesgue measure. However, the marginal distribution of the intrisic parameters $\theta = (x,y,S,\gamma)$ is absolutely continous with respect to $\lambda^{\otimes 4}$, the Lebesgue measure over $\R^4$. Let $p_0^\theta: \R^4\rightarrow \R_+$ be the density of $\theta$.
\begin{equation}
  \forall (x,y,S,\gamma)\in \R^4,~p_0^\theta(x,y,S,\gamma) = \frac{\I\{0\leq x,y\leq L\}\I\{S_1(x)\leq S\leq S_2(x)\}\I\{\gamma_1(y)\leq \gamma\leq \gamma_2(y)\}}{L^2\sigma_S\sigma_\gamma}
\end{equation}

We can therefore simulate the population model by first drawing samples from the distribution $\mu_0$, and finally by solving the differential system (\ref{system_schneider}) using standard numerical methods. In our case, we have used a simple Euler explicit method with a time step of $\Delta t = 0.1$ day. The following table gives the configuration used for the simulation.

\begin{table}[H]
  \begin{center}
    \begin{tabular}{|cccc|}
      \hline
      $L$             & \multicolumn{1}{c|}{1 m}                  & $S_M$          & 1 m                \\
      $S_m$           & \multicolumn{1}{c|}{0.8 m}                & $\gamma_M$     & 1 day$^{-1}$       \\
      $\gamma_m$      & \multicolumn{1}{c|}{0.1 day$^{-1}$}       & $\sigma_S$     & $10^{-2}$ m        \\
      $\sigma_\gamma$ & \multicolumn{1}{c|}{$10^{-2}$ day$^{-1}$} & $s^0$          & 0.3 m              \\
$s_m$           & \multicolumn{1}{c|}{$5.10^{-2}$ m}        & $R_M$          & $\log(S_M/s_m)$    \\
      $\sigma_x$      & \multicolumn{1}{c|}{$L$}                  & $\sigma_r$     & $\log(0.1/s_m)$    \\ \hline
      \multicolumn{2}{|c}{$\Delta t$}                            & \multicolumn{2}{c|}{0.1 day} \\ \hline
    \end{tabular}
  \end{center}
  \caption{Configuration of the parameters for the simulation of the system \ref{system_schneider}}
  \label{table_schneider}
\end{table}

\begin{figure}[ht]
  \begin{center}
    \subfloat[Plant in the middle]{
      \includegraphics[width = 0.6\textwidth]{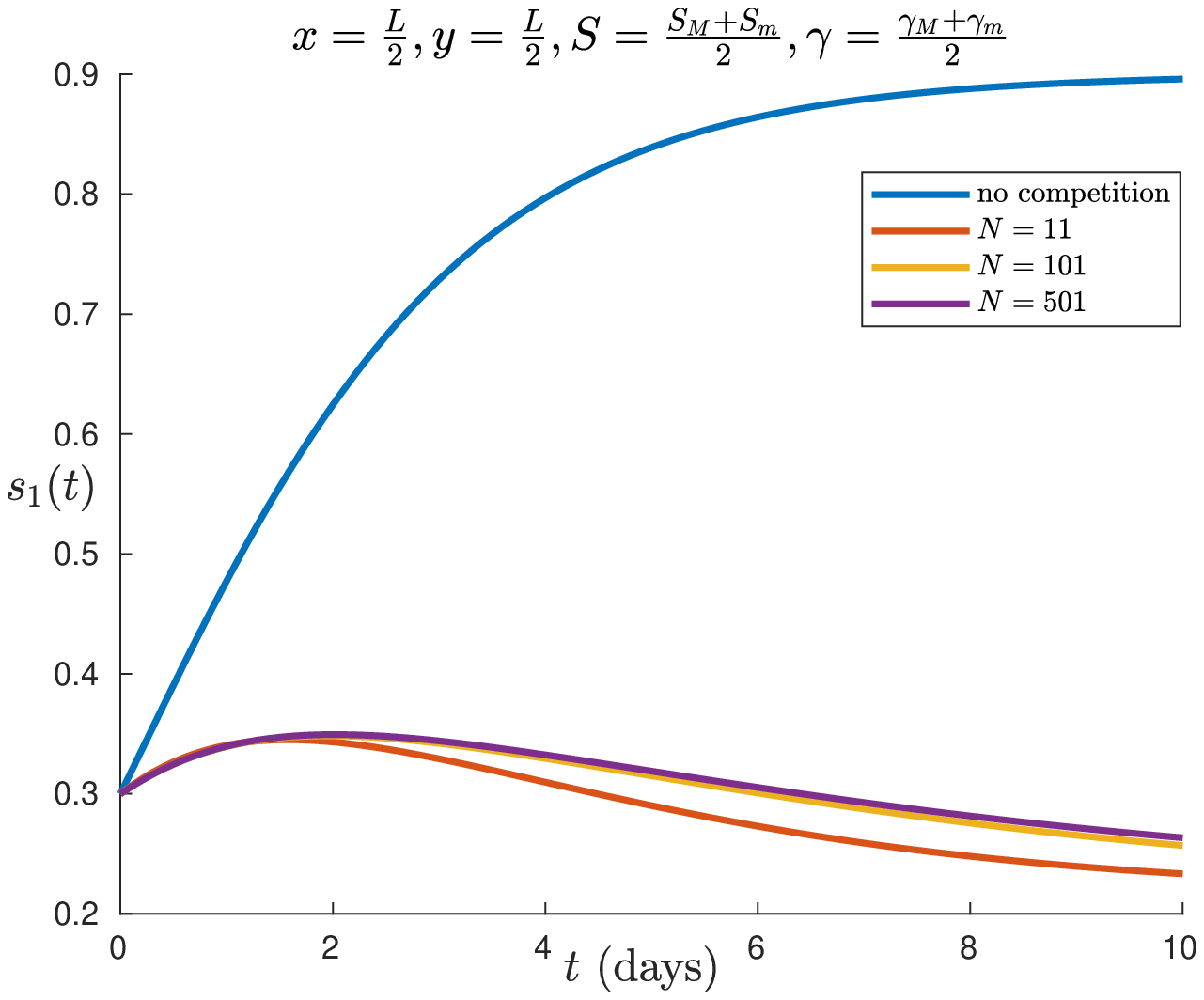}}
    \subfloat[Plant at the boundary]{
      \includegraphics[width = 0.6\textwidth]{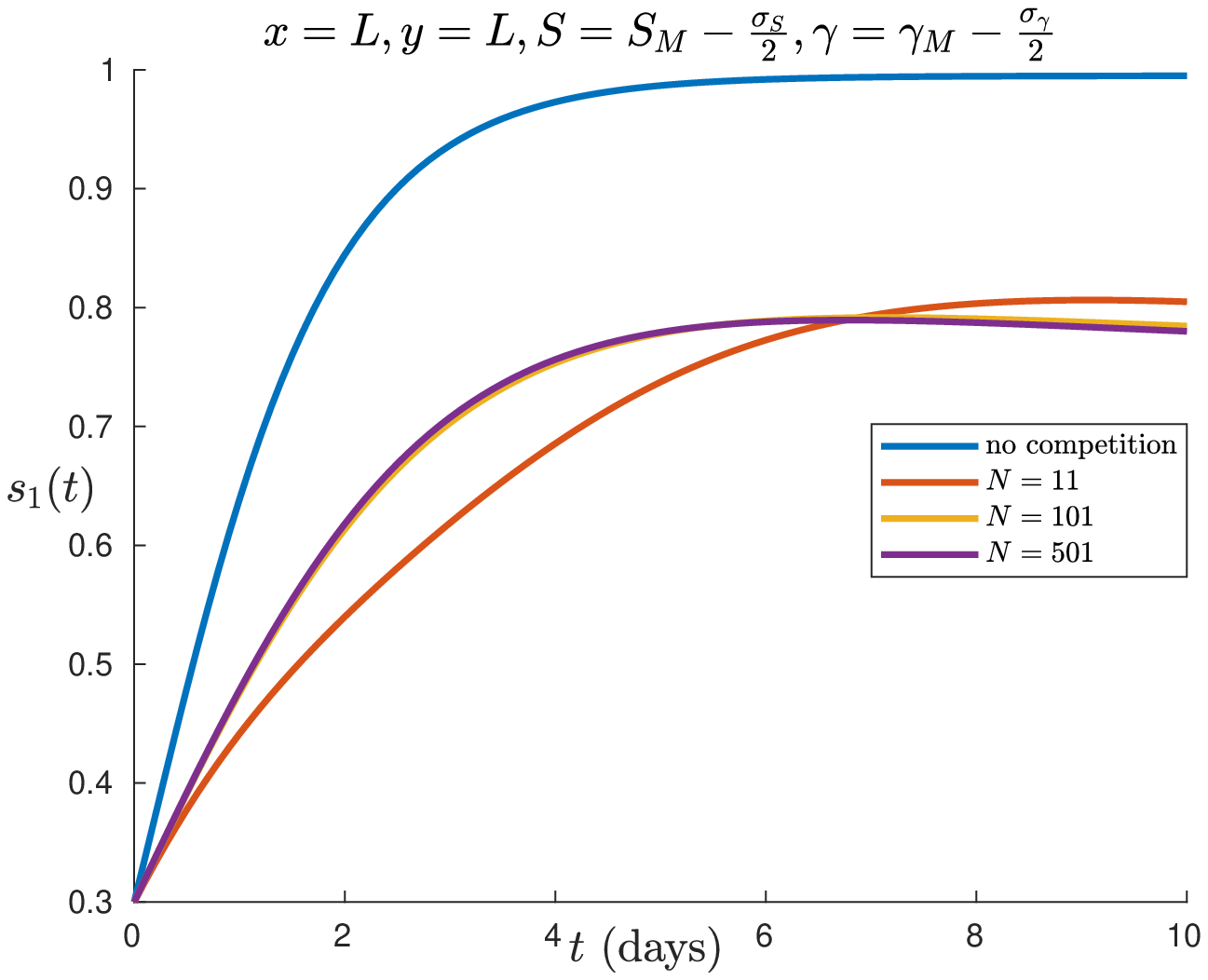}}
  \end{center}
  \caption{Comparison of plant growths for different characteristics and for different situations of competition. The plant in figure (a) is in the middle of the domain $[0;L]^2$, whereas the plant in figure (b) is located at the upper right corner. The plant in figure (a) has a slower growth rate $\gamma$ than the plant (b), and also a smaller equilibrium size $S$ than plant (b). The cases where the plant with $0,10,100,$ and 500 other plants are represented on the same graph.}
  \label{simu_schneider1}
\end{figure}

A visualization of the impact of competition on plant growth is presented in figure \ref{simu_schneider1}. Depending on its position and on its intrisic parameters $S$ and $\gamma$, the response of a plant to competition with the rest of the population can varie significantly. In the middle of the domain $[0;L]^2$, a plant is more subject to competition than a plant at the boundary, since it is surrounded by more competitors. The evolution of the size over the time depends also on the number $N$ of individuals. We can notice that the responses are quite different from $N = 11$ to $N = 101$, but there are very little changes from $N = 101$ to $N = 501$. This convergence constitutes a first visualization of the MFL distribution : as $N$ increases, the finite sample of competitors behaves more and more as a deterministic continuum. The next section gives a formal proof of this statement.

We can consider the change of variable $\displaystyle r = \log\left(\frac{s}{s_m}\right)$, so that the state variable $r$ lies in the vector space $\X = \R$, and $d_\X = 1$. This change of variable is also applied to the initial distribution $\mu_0$, so that the marginal initial distribution of the state is for now on related to $r$ variable $\mu_0^r = \delta_{r^0} = \delta_{\log(s^0/s_m)}$. The parameter space $\Theta$ can be chosen as $\Theta= [0;L]^2\times [s_m;S_M]\times [\gamma_m;\gamma_M]$. The reference vector to define a norm over $\Z$ can be chosen as $z^* = (1,L,L,s_m,\gamma_m)$.
\begin{equation}
  \forall z = (r,x,y,S,\gamma)\in \Z,~|z| = |r| + \frac{|x| + |y|}{L} + \frac{|S|}{s_m} + \frac{|\gamma|}{\gamma_m}
\end{equation}
The interaction function $g$ has the expression of function $g_r$ defined in equation (\ref{definition_gr}). Over $\Z$, all four assumptions are satisfied by function $g$. Possible choices of constants $K_1,K_2,K_3,K_4$ are given below.
\begin{equation}
  \begin{aligned}
    &K_1 = \gamma_M\max(1,R_M) & K_2 = \gamma_M\max\left(1,\frac{1}{4\sigma_r}\right)\\
    & K_3 = \gamma_M\max\left(1,\frac{1}{2\sigma_r}\right)& K_4 = \gamma_M\max\left(2+R_M,1+\frac{2L^2}{\sigma_x^2}\right)
  \end{aligned}
\end{equation}

\section{Derivation of the mean-field limit}

This section follows similar steps as in \cite{Golse2013} to establish the MFL distribution associated to system (\ref{system_micro}). We start by proving that system (\ref{system_micro}) implies a transport equation verified by the empirical measure of the population. From this transport equation, we derive the expression of the MFL transport equation monitoring the dynamics of the MFL distribution. Finally, the connection between the two transport equations is given by Dobrushin stability, which implies also chaos propagation.

\subsection{Properties of the population empirical measure}
\label{sub_empirical_measure}

The system (\ref{system_micro}) has an unique global solution if the interaction function $g$ satisfies assumptions (A1) and (A2). Let $Z_N^0 = (z_i^0)_{1\leq i\leq N} = ((X_i^0,\theta_i))_{1\leq i\leq N}$
be an initial configuration of the system (\ref{system_micro}). We introduce $t\in \R_+\mapsto Z_N(t,Z_N^0) = (z_i(t,Z_N^0))_{1\leq i\leq N} = ((X_i(t,Z_N^0),\theta_i))_{1\leq i\leq N} \in \Z^N$ the global solution of the system. The empirical measure of the population is defined as the map
\begin{equation}
  t\in \R_+\mapsto \mu[t,Z_N^0] = \frac{1}{N}\sum_{i = 1}^N \delta_{z_i(t,Z_N^0)}
  \label{empirical_measure}
\end{equation}
In the above equation, we use the notation $\forall z\in \Z, \delta_z$ is the Dirac distribution centered at $z$, i.e. the distribution of the random variable which is almost surely constant equal to $z$. The empirical measure of the population is a dynamical probability distribution. Sampling this distribution at a fixed time $t$ corresponds to choose an individual uniformly over the population (with probability $\displaystyle \frac{1}{N}$). Interestingly, the empirical measure describes exhaustively the dynamics of the whole population, while remaining in a space $\probaspace(\Z)$, which does not depend of the population size $N$. However, there is a loss of information from vector $Z_N(t,Z_N^0)$, where all individuals are labelled by indices in $\llbracket 1;N\rrbracket$, to the measure $\mu[t,Z_N^0]$ where all individuals are not distinguishable. In other words, a visualization of vector $Z_N^0$ in the phase space $\Z$ would be a cloud of points with all different colors, whereas a visualization of $\mu[t,Z_N^0]$ would be the same cloud of points with a single color. This indistinction of the individuals is a first step towards the mean-field limit, where individuals are punctual parts of a continuum.

Let us characterize the dynamics of $\mu[t,Z_N^0]$ using the system (\ref{system_micro}). We observe the dynamics of a probability measure through its action on test functions, which in our case is the functional space $\C^1_0(\R_+\times \Z\rightarrow \R)$.
\begin{equation}
  \begin{aligned}
    &\C^1_0(\R_+\times \Z\rightarrow \R) = \left\{\varphi :\R_+\times \Z \rightarrow \R \text{ continuously differentiable}|\right.\\
    &\left.\lim_{|z|\rightarrow +\infty} |\varphi(t,z)|+\left|\frac{\partial\varphi}{\partial z}(t,z)\right| + \left|\frac{\partial\varphi}{\partial t}(t,z)\right| = 0\right\}
  \end{aligned}
\end{equation}
In particular, the test functions considered in this article are bounded over their domain, and have bounded derivatives. We call action of $\mu[t,Z_N^0]$ on a test function $\varphi \in \C^1_0(\R_+\times \Z\rightarrow \R)$ the dual pairing of $\mu[t,Z_N^0]$ and $\varphi$ or the expectaction of random variable $\varphi(t,z)$ where $z$ is a random variable of distribution $\mu[t,Z_n^0]$.
\begin{equation}
  \expect_{\mu[t,Z_N^0]}(\varphi(t,z)) = \int_\Z \varphi(t,z)\mu[t,Z_N^0](\intd z) = \frac{1}{N}\sum_{i=1}^N\varphi(t,z_i(t,Z_N^0))
\end{equation}
The time evolution of $t\in \R_+\mapsto \mu[t,Z_N^0]$ can be studied by considering the differential equation satisfied by $t\in \R_+\mapsto \displaystyle \int_\Z\varphi(t,z)\mu[t,Z_N^0](\intd z)$ for any test function $\varphi$. So let us express the derivative $\displaystyle \frac{\intd}{\intd t} \int_\Z \varphi(t,z)\mu[t,Z_N^0](\intd z)$ as an action of $\mu[t,Z_N^0]$ on some function depending on $\varphi$ and on the interaction function $g$.
\begin{equation*}
    \begin{aligned}
      &\frac{\intd }{\intd t} \int_\Z \varphi(t,z)\mu[t,Z_N^0](\intd z) = \frac{1}{N}\sum_{i=1}^N\frac{\partial \varphi}{\partial t}(t,z_i(t,Z_N^0)) + \frac{\partial \varphi}{\partial X}(t,z_i(t,Z_N^0))^\T\frac{\intd X_i}{\intd t}(t,Z_N^0)\\
      & = \int_\Z\frac{\partial \varphi}{\partial t}(t,z)\mu[t,Z_N^0](\intd z) +  \frac{1}{N(N-1)}\sum_{i=1}^N\sum_{j\neq i}\frac{\partial \varphi}{\partial X}(t,z_i(t,Z_N^0))^\T g(z_i(t,Z_N^0),z_j(t,Z_N^0))\\
      &\frac{1}{N(N-1)}\sum_{i=1}^N\sum_{j\neq i}\frac{\partial \varphi}{\partial X}(t,z_i(t,Z_N^0))^\T g(z_i(t,Z_N^0),z_j(t,Z_N^0)) =\\
      &\int_\Z \frac{\partial \varphi}{\partial X}(t,z)^\T\left(\frac{N}{N-1}\int_\Z g(z,z')\mu[t,Z_N^0](\intd z')-\frac{1}{N-1}g(z,z)\right)\mu[t,Z_N^0](\intd z)
    \end{aligned}
  \end{equation*}
\begin{equation}
  \begin{aligned}
    &\frac{\intd }{\intd t} \int_\Z \varphi(t,z)\mu[t,Z_N^0](\intd z) = \\
    &\int_\Z\left(\frac{\partial \varphi}{\partial t}(t,z) +  \frac{\partial \varphi}{\partial X}(t,z)^\T\left(\frac{N}{N-1}\int_\Z g(z,z')\mu[t,Z_N^0](\intd z')-\frac{1}{N-1}g(z,z)\right)\right)\mu[t,Z_N^0](\intd z)
  \end{aligned}
  \label{weak_transport_micro}
\end{equation}
In the above equation, we can interpret the term
\begin{equation}
  \forall z\in \Z,~\G_N(\mu[t,Z_N^0],z) = \frac{N}{N-1}\int_\Z g(z,z')\mu[t,Z_N^0](\intd z')-\frac{1}{N-1}g(z,z)
  \label{velocity_micro}
\end{equation}
as the velocity field associated to the system (\ref{system_micro}), i.e. the one assigning to each individual its velocity according to the current state of the whole population\footnote{Similar developments can be found in \cite{Carrillo2010}}. $\G_N$ is said to be a \textit{non-local} velocity field, because it depends on the probability measure describing the state of the population, $\mu[t,Z_N^0]$ in this case. The velocity field can be associated to a conservative transport equation, having a formulation quite similar, in its principle at least, to Vlasov equations, where the velocity field depends on the unknown density (see \cite{Golse2003}, section 1.1.1).
\begin{equation}
  \begin{aligned}
    &\frac{\partial f}{\partial t}(t,z) + \Div_X\left(f(t,z)\left(\frac{N}{N-1}\int_\Z g(z,z')f(t,z')\lambda^{\otimes d_z}(\intd z') - \frac{g(z,z)}{N-1}\right) \right) = 0\\
    &\text{or } \frac{\partial f}{\partial t}(t,z) + \Div_X\left(f(t,z)\G_N\left(f(t,\cdot)\lambda^{\otimes d_z},z\right)\right) = 0
  \end{aligned}
  \label{transport_micro}
\end{equation}
where $\Div_X$ is the divergence operator with respect to state variable $X$, i.e. for any continuously differentiable map $F:\X\rightarrow \X$, $\Div_X F(X) = \displaystyle \sum_{i = 1}^{d_X}\frac{\partial F_i(X)}{\partial X_i}$, and $f(t,\cdot)\lambda^{\otimes d_Z}$ is the probability measure of density $f(t,\cdot):\Z\rightarrow \R_+$. The equation (\ref{weak_transport_micro}) is a weak formulation of equation (\ref{transport_micro}), which is formally defined in definition \ref{measure_solution_def}. As the weak formulation deals with trajectories taking values in the space of probability measures, we need to introduce the Wasserstein distance to quantify the regularity of these trajectories.

\begin{definition}
  Let $\mu_1,\mu_2\in \probaspace_1(\Z)$. Let $\Pi(\mu_1,\mu_2)$ the set of couplings of $\mu_1$ and $\mu_2$, i.e. the set of probability distributions having its first and second marginals equal to $\mu_1$ and $\mu_2$ respectively.
  \begin{equation*}
    \Pi(\mu_1,\mu_2) = \left\{\pi\in \probaspace_1(\Z^2)\left|\mu_1 = \int_\Z\pi(.,\intd z_2),\mu_2 = \int_\Z\pi(\intd z_1,.)\right.\right\}
  \end{equation*}
  The Wasserstein distance of first order between $\mu_1$ and $\mu_2$ is defined by
  \begin{equation}
    W_1(\mu_1,\mu_2) = \inf_{\pi\in \Pi(\mu_1,\mu_2)}\int_{\Z^2}|z_1-z_2|\pi(\intd z_1,\intd z_2)
  \end{equation}
  or, equivalently, the Wasserstein distance of first order has a dual representation (\cite{Kantorovich1958})
  \begin{equation}
    \begin{aligned}
      W_1(\mu_1,\mu_2) = \sup_{\varphi \in \C_L(\Z),\mathrm{Lip}(\varphi)\leq 1}\left|\int_\Z \varphi(z)\mu_1(\intd z)-\int_\Z\varphi(z)\mu_2(\intd z)\right|
    \end{aligned}
  \end{equation}
   with $\C_L(\Z)$ being the space of Lipschitz-continuous functions over $\Z$, taking values in $\R$, and $\mathrm{Lip}(\varphi)$ being the Lipschitz constant of $\varphi\in \C_L(\Z)$.
  \label{def_wasserstein}
\end{definition}

\begin{definition}
  \label{measure_solution_def}
  Let $\G : (\mu,z)\in\probaspace_1(\Z)\times \Z\rightarrow \X$ a non-local velocity field and $\mu_0\in \probaspace_1(\Z)$ a probability measure having first order moment, i.e. $\displaystyle \int_\Z |z|\mu_0(\intd z)<+\infty$. We say that the trajectory $t\in \R_+\mapsto \mu[t]\in \probaspace_1(\Z)$ is a measure solution of the transport equation of velocity field $\G$ and of initial condition $\mu_0$ if
  \begin{enumerate}
  \item the trajectory $t\in \R_+\mapsto \mu[t]\in \probaspace_1(\Z)$ is continuous for the metric $W_1$.
  \item for all test function $\varphi$, for all time $t\in \R_+$
    \begin{equation*}
      \begin{aligned}
        &\int_\Z \varphi(t,z)\mu[t](\intd z) - \int_\Z \varphi(0,z)\mu_0(\intd z) = \\
        &\int_0^t\int_\Z \left(\frac{\partial \varphi}{\partial t}(\tau,z) + \frac{\partial \varphi}{\partial X}(\tau,z)^\T\G(\mu[\tau],z)\right)\mu[\tau](\intd z)\intd \tau
      \end{aligned}
    \end{equation*}
    
  \end{enumerate}
\end{definition}
\begin{proposition}
Let $g:\Z^2\rightarrow \X$ satisfying assumptions (A1) and (A2) and $Z_N^0 = (z_i^0)_{1\leq i\leq N} = ((X_i^0,\theta_i))_{1\leq i\leq N}\in \Z^N$. Then the empirical measure $t\in \R_+\mapsto \mu[t,Z_N^0]$, defined in equation (\ref{empirical_measure}) is a measure solution to the transport equation of velocity field $\G_N$, defined in equation (\ref{velocity_micro}), and of initial condition $\mu[0,Z_N^0] = \displaystyle \frac{1}{N}\sum_{i=1}^N\delta_{z_i^0}$.
\end{proposition}
This proposition summarizes the equation (\ref{weak_transport_micro}). It is also necessary to check the continuity of the trajectory $t\in \R_+\mapsto \mu[t,Z_N^0]\in \probaspace_1(\Z)$ is continuous for the metric $W_1$. This continuity is directly given by the continuity of the solution of the system (\ref{system_micro}) (see appendix section \ref{proof_prop3}).

The transport equation (\ref{weak_transport_micro}) satisfied by the empirical measure leads to the transport equation describing the dynamics of the population with an infinite number of individuals by taking the limit $N \rightarrow +\infty$. The resulting equation, obtained informally, is referred as the MFL transport equation, and its eventual solution is the MFL distribution. Subsection \ref{sub_macro} solves the MFL transport equation and proves the existence and uniqueness of the MFL distribution for system (\ref{system_micro}). Subsection \ref{sub_dobrushin} studies different aspects of the convergence towards the MFL distribution.

\subsection{Study of the mean-field equation}
\label{sub_macro}

This subsection gives a characterization of the MFL distribution as the unique solution of a non-local transport equation obtained as the limite case of transport equation (\ref{weak_transport_micro}). Let us assume informally that, for some metric over the space of probability measures (namely the Wasserstein distance, see subsection \ref{sub_dobrushin}), the empirical measure $\mu[t,Z_N^0]$ has a limit $\mu[t]$ when $N\rightarrow +\infty$. Then it is reasonable to think that, for some other metric, the velocity field $\G_N(\mu[t,Z_N^0],\cdot)$ converges towards a velocity field depending on $g$ and $\mu[t]$. The only expression this velocity field can reasonably have, when $N\rightarrow +\infty$ in equation (\ref{velocity_micro}), is
\begin{equation}
  \forall z\in \Z,~\G(\mu[t],z) = \int_\Z g(z,z')\mu[t](\intd z')
  \label{velocity_g}
\end{equation}
So if the MFL distribution exists, it has to be a measure solution of the transport equation of velocity field $\G$ and of initial condition $\mu_0$, as it is reminded in subsection \ref{sub_dobrushin} that $\mu[0,Z_N^0]$ converges towards $\mu_0$ for the Wasserstein distance. The MFL distribution is therefore an eventual trajectory $t\in \R_+\mapsto \mu[t]\in \probaspace_1(\Z)$, continuous for the metric $W_1$, such that for all test function $\varphi$ and for all time $t\in \R_+$
\begin{equation}
  \begin{aligned}
    &\int_\Z \varphi(t,z)\mu[t](\intd z) - \int_\Z \varphi(0,z)\mu_0(\intd z) = \\
    &\int_0^t\int_\Z \left(\frac{\partial \varphi}{\partial t}(\tau,z) + \frac{\partial \varphi}{\partial X}(\tau,z)^\T\int_\Z g(z,z')\mu[t](\intd z')\right)\mu[\tau](\intd z)\intd \tau
  \end{aligned}
  \label{weak_transport_macro}
\end{equation}
Drawing largely on \cite{Golse2013}, we would like to use the characteristic flow method to prove the existence and uniqueness of the solution to the equation (\ref{weak_transport_macro}). The characteristic flow method is a classical idea to study a transport equation : the transport PDE describes the dynamics, while the characteristic flow equation describes the motion of a single particle, subject to the same velocity field. Let $t\in \R_+\mapsto X_\infty(t,X,\theta)\in \X$ be the trajectory of a particle immersed in velocity field $\G$, of initial configuration $(X,\theta)\in \Z$.
\begin{equation}
  \left\{\begin{array}{l}
  X_\infty(0,X,\theta) = X\\
  \displaystyle \forall t\in \R_+,~\frac{\partial X_\infty}{\partial t}(t,X,\theta) = \G(\mu[t],X_\infty(t,X,\theta),\theta)
  \end{array}\right.
  \label{equ_X_macro1}
\end{equation}
As $\mu[t]$ is unknown, we cannot evaluate the derivative $\displaystyle \frac{\intd X_\infty}{\intd t}$, except at $t = 0$, when $\mu[0] = \mu_0$. However, there is a strong connection between the MFL distribution $\mu[t]$ and the flow $X_\infty(t,\cdot)$, as $\mu[t]$ describes the state of a population which is composed of infinite number of particles $(X_\infty(t,X',\theta'),\theta')$, whose initial configuration is given by $\mu_0$. In other words, we can look at the interaction with the rest of the population not as an average over all the possible states at time $t$, as it is the case in equation (\ref{equ_X_macro1}), but as an average over all initial configurations.
\begin{equation}
  \left\{\begin{array}{l}
  X_\infty(0,X,\theta) = X\\
  \displaystyle \forall t\in \R_+,~\frac{\partial X_\infty}{\partial t}(t,X,\theta) = \int_\Z g(X_\infty(t,X,\theta),\theta,X_\infty(t,X',\theta'),\theta')\mu_0(\intd X',\intd \theta')
  \end{array}\right.
\end{equation}
There is only a single unknown in the above functional equation, which is the flow $X_\infty:\R_+\times \Z\rightarrow \X$. The next theorem shows that this characteristic flow is well defined.

\begin{theorem}
  \label{thm_flow_X}
  Let $\mu_0\in \probaspace_2(\Z)$ a probability measure having second order moment, i.e. $\displaystyle \int_\Z |z|^2\mu_0(\intd z) < +\infty$, and $g:\Z^2\rightarrow \X$ satisfying assumptions (A1) and (A2). Then there exists an unique flow such that
  \begin{enumerate}
  \item $\forall t\in \R_+,~\forall (X,\theta)\in \Z,~\displaystyle \int_\Z |g(X_\infty(t,X,\theta),\theta,X_\infty(t,X',\theta'),\theta')|\mu_0(\intd X',\intd \theta') < +\infty$
  \item $\forall (X,\theta)\in\Z,~t\in\R_+\mapsto X_\infty(t,X,\theta)$ is continuously differentiable.
  \item $\forall (X,\theta)\in \Z$,
    \begin{equation}
      \begin{aligned}
        &\left\{\begin{array}{l}
        X_\infty(0,X,\theta) = X\\
        \displaystyle \forall t\in \R_+,~\frac{\partial X_\infty}{\partial t}(t,X,\theta) = \int_\Z g(X_\infty(t,X,\theta),\theta,X_\infty(t,X',\theta'),\theta')\mu_0(\intd X',\intd \theta')
        \end{array}\right.
      \end{aligned}
      \label{equ_flow_X}
    \end{equation}
  \end{enumerate}
\end{theorem}
The above functional equation can be seen as a continuous version of system (\ref{system_micro}). Formally, this equation is a differential equation with an initial condition being a probability measure and having trajectories in vector space $\X$. This theorem can therefore be proved by following exactly the same steps as for Cauchy-Lipschitz theorem with traditional differential equations. A common proof of Cauchy-Lipschitz theorem is based on fixed point theorem within a functional Banach space. The functional space where the flow solution $X_\infty$ lies must be complete, as the fixed point theorem recquires the convergence of any Cauchy sequence. The next lemma introduces the functional space used in the proof of theorem 1.

\begin{lemma}(\cite{Golse2013})
  Let $\Y$ be the functional space defined by
  \begin{equation}
    \Y = \left\{f\in \C^0(\Z\rightarrow \X)\left|\sup_{z\in \Z}\frac{|f(z)|}{1+|z|} <+\infty\right.\right\}
  \end{equation}
  Then $\Y$ is a Banach space for the metric $f\in \Y\mapsto \|f\|_\Y = \displaystyle \sup_{z\in \Z}\frac{|f(z)|}{1+|z|}$.
\end{lemma}
The proof of the above theorem is given in appendix \ref{proof_thm1}, and it follows the same steps than the proof of Cauchy-Lipschitz theorem for ordinary differential equations: local existence and uniqueness, existence of a maximal solution, uniqueness of the maximal solution and finally definition over $\R_+$ of the maximal solution. The assumption (A1) on $g$ is important to ensure that the flow solution is defined over $\R_+$ and also the stability within the functional space $\Y$. Besides, the control of the Lipschitz factor in assumption (A2) can be relaxed, as long as the Lipschitz factor is compensated by the initial distribution. For instance, if the Lipschitz factor in (A2) is $K_2(1+|X|^n+|X'|^n)$ for some $n > 0$ instead of $K_2(1+|X|+|X'|)$, then similar reasoning can be carried out to obtain the existence and uniqueness of $X_\infty$, if $\mu_0$ is chosen in $\probaspace_{n+1}(\Z)$, i.e. such that $\displaystyle \int_\Z |z|^{n+1}\mu_0(\intd z)<+\infty$.

In the case of Schneider model, the flow solution $t\in\R_+\mapsto r_\infty(t,.)$ satisfies the following equation
\begin{equation}
  \begin{aligned}
    &\forall (r,x,y,S,\gamma)\in \Z = \R\times \Theta,\\
    &\left\{\begin{array}{l}
    r_\infty(0,r,\theta) = r\\
    \displaystyle \forall t\in \R_+,~\frac{\partial r_\infty}{\partial t}(t,r,\theta) = \int_\Z g_r(r_\infty(t,r,\theta),\theta,r_\infty(t,r',\theta'))\mu_0(\intd r',\intd \theta')
    \end{array}\right.\\
    &g_r(r_1,(\vec x_1,S_1,\gamma_1),r_2,(\vec x_2,S_2,\gamma_2)) = \gamma_1\left(\log\left(\frac{S_1}{s_m}\right)\left(1-C_r(r_1,r_2,|\vec x_1-\vec x_2|)\right)-r_1\right)\\
    &C_r(r_1,r_2,|\vec x_1-\vec x_2|) = \frac{r_2}{2R_M\displaystyle\left(1+\frac{|\vec{x}_1-\vec{x}_2|^2}{\sigma_x^2}\right)}\left(1+\tanh\left(\frac{r_2-r_1}{\sigma_r}\right)\right)
  \end{aligned}
\end{equation}
If $\mu_0$ is the distribution defined in equation (\ref{eq_mu0_schneider}), we have for all time $t\in \R_+$ and for all $(r,\theta)\in \Z$,
\begin{equation}
  \frac{\partial r_\infty}{\partial t}(t,r,\theta) = \int_\Theta g_r(r_\infty(t,r,\theta),\theta,r_\infty(t,r^0,\theta'),\theta')p_0^\theta(\theta')\lambda^{\otimes 4}(\intd \theta')
  \label{eq_flow_r}
\end{equation}
where $r^0 = \log\left(\frac{s^0}{s_m}\right)$. This relation holds because the marginal distribution of the initial state is a Dirac distribution centered at $r^0$.

The characteristic flow (\ref{equ_flow_X}) leads to the unique solution of the mean-field transport equation (\ref{weak_transport_macro}), which appears as the pushforward probability measure of the initial distribution $\mu_0$ by the map $X_\infty(t,.)$. This result is used qualitatively to derive equation (\ref{equ_flow_X}) from equation (\ref{equ_X_macro1}).
\begin{corollary}
  Let $\mu_0\in \probaspace_2(\Z)$, g satisfying assumptions (A1), (A2) and (A3), and $z^0 = (X^0,\theta)$ a random variable of distribution $\mu_0$. Then the unique measure-solution to the transport equation (\ref{weak_transport_macro}) is $t\in \R_+\mapsto \mu[t]\in \probaspace_2(\Z)$ where for all $t\in \R_+$, $\mu[t]$ is the probability distribution of $z^t = (X_\infty(t,z^0),\theta)$.
\end{corollary}
The proof of this corollary (appendix, section \ref{proof_thm2}) is a generalization of the method of characteristic flows, classically used in the field of hyperbolic PDE. The proof that the pushforward measure is effectively a measure solution of the mean-field transport equation (\ref{weak_transport_macro}) is mainly based on the change-of-variable formula, stating that for every test function $\phi$, we have $\displaystyle \int_\Z \phi(t,z)\mu[t](\intd z) = \int_\Z\phi(t,X_\infty(t,X,\theta),\theta)\mu_0(\intd X,\intd\theta)$. The continuity of the trajectory $t\mapsto \mu[t]$ is therefore implied by the continuity of $t\mapsto \| X_\infty(t,.)\|_\Y$. Besides, the proof of uniqueness requires additional assumptions on the regularity of the interaction function $g$, namely assumption (A3). It was added for the sake of brevity, but it seems that this assumption can be avoided by adding technical developments using an argument of density of the space of test functions. This regularity enables to prove that the flows $X^\nu$ associated to the velocity field $\G(\nu[t],z)$, where $t\mapsto \nu[t]$ is a fixed measure-trajectory, are continuously differentiable with respect to their arguments. This implies also the regularity of the solution $(t,z)\mapsto \varphi(t,z)$ to the following transport equation
\begin{equation*}
  \left\{\begin{array}{l}
  \varphi(0,z) = \varphi_0(z)\\
  \displaystyle \frac{\partial \varphi}{\partial t}(t,z) + \frac{\partial \varphi}{\partial X}(t,z)^\T\G(\nu[t],z) = 0
  \end{array}\right.
\end{equation*}
As $\varphi$ is regular, it can be used as a test function itself. If $t\mapsto \nu[t]$ is a measure-solution of (\ref{weak_transport_macro}), it can be shown that $t\mapsto \int_\Z \varphi(t,z)\nu[t](\intd z)$ is constant. This implies in particular that $\nu[t]$ is the pushforward measure of $\mu_0$ by the map $X^\nu(t,.)$. It follows that $X^\nu$ and $X_\infty$ satisfies the same equation (\ref{equ_flow_X}) and we conclude by uniqueness of the characteristic flow provided by theorem 1.

Alternative proofs of existence and uniqueness of the MFL distribution can be found in \cite{Lagoutiere2017} or in \cite{Bolley2011}, with weaker assumptions made on the velocity field. \cite{Lagoutiere2017} uses Filipov characteristics and compactness arguments to solve the transport equation associated to a bounded velocity field with a finite set of discontinuities. The velocity field considered in \cite{Bolley2011} is not globally Lipschitz continuous, as in our case. The resolution of an equation similar to (\ref{equ_X_macro1}) follows an iterative procedure : the measure trajectory $t\mapsto \mu^n[t]$ is fixed at iteration $n$, and is used to compute the characteristic flow $X^{\mu^n}$, by solving a standard differential equation ; the distribution $t\mapsto \mu^{n+1}[t]$ chosen at the next iteration is the pushforward measure of $\mu_0$ by the characteristic flow $X^{\mu^n}$.

Let us now consider the case where the initial distribution $\mu_0$ is absolutely continuous with respect to the Lebesgue measure $\lambda^{\otimes d_z}$. We denote by $p_0:\Z\rightarrow \R_+$ its associated probability density. $p_0$ can be factorized in two terms using the chain rule.
\begin{equation*}
  \forall (X,\theta)\in \Z,~p_0(X,\theta) = p_0^{X|\theta}(X|\theta)p_0^\theta(\theta)
\end{equation*}
It follows from proposition 4 that, in this case, the MFL distribution $\mu[t]$ is absolutely continuous for all time $t\in \R_+$, and that the associated density $p_t:\Z\rightarrow \R_+$ is given by change-of-variable formula.
\begin{equation*}
  \forall (X,\theta)\in \Z,~\forall t\in \R_+,~p_t(X,\theta) = p_0^{X|\theta}\left(X_\infty^{-1}(t,X,\theta)|\theta\right)p_0^\theta(\theta)\det\left(\frac{\partial X_\infty^{-1}}{\partial X}(t,X,\theta)\right)
\end{equation*}
In the above equation, for all $t\in \R_+$ and $\theta\in \Theta$, $X\in \X\mapsto X_\infty^{-1}(t,X,\theta)$ is the inverse function of $X\in \X\mapsto X_\infty(t,X,\theta)$ and $\displaystyle \frac{\partial X_\infty^{-1}}{\partial X}(t,X,\theta)$ is the Jacobian matrix of this function.

The fact that $X_\infty(t,.,\theta)$ is a one-to-one map from $\X$ to $\X$ is a consequence of the flow property. This property considers a variation of the initial time in equation (\ref{equ_flow_X}). There exists an unique map $(t,t_0,z)\mapsto X_\infty(t,t_0,z)$ satisfying 
\begin{equation*}
  \left\{\begin{array}{l}
  X_\infty(t_0,t_0,X,\theta) = X\\
  \displaystyle \forall t\in \R_+,~\frac{\partial X_\infty}{\partial t}(t,t_0,X,\theta) = \int_\Z g(X_\infty(t,t_0,X,\theta),\theta,X',\theta')\mu[t](\intd X',\intd \theta')
  \end{array}\right.
\end{equation*}
By unicity, we have that for all $t,t_0 \in \R_+$ and $(X,\theta)\in \Z$, $X_\infty(t,t_0,X_\infty(t_0,t,X,\theta),\theta) = X$. It follows that $X_\infty^{-1}(t,X,\theta) = X_\infty(0,t,X,\theta)$. The differentiability of the map $X\in \X\mapsto X_\infty^{-1}(t,X,\theta)$ is a consequence of assumption (A3).

In the case of Schneider model, the MFL distribution $\mu[t]$ is the law of the random variable $(r_\infty(t,r^0,\theta),\theta)$ with $\theta\sim p_0^\theta$. Therefore $\mu[t]$ is entirely determined by the map $(t,\theta)\in \R_+\times \theta\mapsto r_\infty(t,r^0,\theta)$. This is due to the fact that the initial state is constant over the whole population, equal to $r^0$. We can notice that this situation seems much simpler than the case where the marginal density of the initial state is absolutely continuous : here, we only need to compute the characteristic flow $r_\infty$, and we do not need to compute its inverse function and its derivative. Section 4 describes a methodology to approximate the characterstic flow $r_\infty$ and therefore to sample the MFL distribution in the specific case of Schneider model.

\subsection{Dobrushin stability and propagation of chaos}
\label{sub_dobrushin}

The relation between the microscopic level, represented by the empirical measure of the population, and the mean-field level, represented by solution of (\ref{weak_transport_macro}), is mainly based on a convergence of the initial empirical distribution $\mu[0,Z_N^0]$ towards the initial distribution $\mu_0$ as N $\rightarrow +\infty$ and on the fact that this results can be extended at all time $t\in \R_+$, i.e. the same type of convergence is verified by $\mu[t,Z_N^0]$ towards $\mu[t]$. The convergence discussed here is the one associated to the metric $W_1$, whose expression is recalled in definition \ref{def_wasserstein}, and which metrizes the weak convergence in the space of probability distribution (see corollary 6.13 in \cite{Villani2008}). According to \cite{Varadarajan1958}, if $(z_n^0)_{n\in \N}$ is a sequence of independent random variables of distribution $\mu_0$, and if for all $N > 1$ $Z_N^0 = (z_1^0,...,z_N^0)$, we have that $\displaystyle \lim_{N\rightarrow+\infty}W_1\left(\mu[0,Z_N^0],\mu_0\right) = 0$ almost surely in $\proba$. This means that there exists $\Omega^*\in \mathcal{F}$ such that $\proba(\Omega^*)=1$ and for all $\omega \in \Omega^*$, we have for all $\phi\in C_L(\Z)$ any Lipschitz continuous function that
\begin{equation}
  \lim_{N\rightarrow +\infty}\frac{1}{N}\sum_{i=1}^N\phi(z_i^0(\omega)) = \int_\Z \phi(z)\mu_0(\intd z)
\end{equation}
A concise proof of this result can also be found in \cite{Golse2013} (theorem 3.3.5).  This result is a consequence of the strong law of large numbers and of the fact that the space of continuous functions with compact support over $\R^{d_z}$ is separable.
The rate of convergence of the random variable $W_1(\mu[0,Z_N^0],\mu_0)$, along with  Wasserstein distance of higher orders, is a well-documented topic in the literature. \cite{Dudley1969} stated that in the case where $\mu_0$ is absolutely continuous with respect to the Lebesgue measure, i.e. can be associated to a probability density $f_0:\Z\rightarrow \R_+$, and if $d_\Z \geq 2$, then there exists a constant $C(\mu_0) >0$ such that for all $N\in \N^*$, $\expect_{Z_N^0\sim \mu_0^{\otimes N}}W_1\left(\mu[0,Z_N^0],\mu_0\right)\leq C(\mu_0)N^{-1/d_\Z}$. In the example of the Schneider model, the chosen initial density described in equation (\ref{eq_mu0_schneider}) is not absolutely continuous with respect to the Lebesgue measure over $\Z$, since the marginal distribution $\mu_0^r$ of the variable $r^0$ is reduced to a Dirac distribution $\delta_{r^0}$, representing the fact hat all plants have the same size $s^0 = s_me^{r^0}$ initially. However, the marginal $\mu_0^\theta$ of variable $\theta$ is associated to a density over $\Theta$, so the upper-bound of Dudley can be rewritten as $\expect_{Z_N^0\sim \mu_0^{\otimes N}}W_1\left(\mu[0,Z_N^0],\mu_0\right)\leq C(\mu_0^\theta)N^{-1/d_\Theta} = C(\mu_0^\theta)N^{-1/4}$. Faster convergence rates can be obtained in the case where probability measure $\mu_0$ is less regular. We can quote notably \cite{Weed2017} in the case where $\mu_0$ is compactly-supported, and \cite{Lei2018} for a generalization to unbounded metric spaces. 

If the random variable $z^0\sim \mu_0$ has at least one of its component with a probability density, then the weak and almost sure convergence of $\mu[0,Z_N^0]$ towards $\mu_0$ means visually that the point cloud $(z_i^0)_{1\leq i\leq N}$ is more and more alike the continuous set represented by measure $\mu_0$. In what follows, the argument of Dobrushin stability (see proposition 4 in \cite{Dobrushin1979} or theorem 3.3.3 in \cite{Golse2013}) is used to prove that for all time $t\in \R_+$ $\displaystyle \lim_{N\rightarrow +\infty} W_1(\mu[t,Z_N^0],\mu[t]) = 0$ almost surely.
\newcommand{\expectation}[2]{\mathbb{E}_{#1\sim #2}}

\begin{theorem}
  Let $\mu_0\in \probaspace(\Z)$ be a probability measure having a compact support, such that the support is included in $\mathcal{B}(0,R_0) = \{z\in \Z | |z| \leq R_0\}$ for some $R_0>0$. Let $(z_n^0)_{n\in \N^*}$ be a sequence of independent random variables of distribution $\mu_0$ and $\forall N >1, Z_N^0 = (z_1^0,...,z_N^0)$. There exists $\Omega^* \in \F$ such that $\proba(\Omega^*) = 1$ and such that $\forall \omega\in \Omega^*,~\displaystyle \lim_{N\rightarrow +\infty}W_1(\mu[0,Z_N^0(\omega)],\mu_0) = 0$. We introduce $\mu[t]$ the solution of problem (\ref{weak_transport_macro}) for an interaction function $g$ satisfying assumptions (A1), (A2), (A3) and (A4). Then we have
  \begin{equation}
    \forall t\in \R_+,~\forall \omega\in \Omega^*,~\lim_{N\rightarrow +\infty}W_1(\mu[t,Z_N^0(\omega)],\mu[t]) = 0
  \end{equation}
  \label{thm_dobrushin}
\end{theorem}

The argument of Dobrushin consists in deriving an upper bound of $W_1(\mu[t,Z_N^0],\mu[t])$ depending on $W_1(\mu[0,Z_N^0],\mu_0)$, which holds for all initial configuration $Z_N^0\in \Z^N$. The derivation of the upper bound is exactly a generalization of Grönwall lemma to characteristic flows of the type of $X_\infty$, i.e. solutions of \textit{differential equations} taking values in space $\X$ and having as initial condition a probability measure over $\Z$. Indeed, we can prove that for all initial configuration $Z_N^0\in \Z^N$, we have
\begin{equation}
  \begin{aligned}
    &\forall t\in \R_+,~W_1(\mu[t,Z_N^0],\mu[t])\leq \exp(F_{\mu_0}(t)+\epsilon_1(t,\mu[t,Z_N^0]))\left(W_1(\mu[0,Z_N^0],\mu_0)\phantom{\frac{1}{N-1}\int_0^t}\right.\\
    &\left.+ \frac{1}{N-1}\int_0^t(E_{\mu_0}(\tau)+\epsilon_2(\tau,\mu[\tau,Z_N^0]))\exp(-F_{\mu_0}(\tau)-\epsilon_1(\tau,Z_N^0))\intd \tau\right)
  \end{aligned}
  \label{ineq_dobrushin}
\end{equation}
The functions $E_{\mu_0}$ and $F_{\mu_0}$ appearing in the previous inequality depend on the first and second moments $M^1_{\mu_0} = \displaystyle \int_\Z|z|\mu_0(\intd z)$ and $M^2_{\mu_0} = \displaystyle \int_\Z |z|^2\mu_0(\intd z)$ of the distribution $\mu_0$. The other functions appearing in the inequality (\ref{ineq_dobrushin}) are such that $\forall t\in \R_+,\forall \omega\in \Omega^*,~\displaystyle \lim_{N\rightarrow +\infty}\epsilon_1(t,\mu[t,Z_N^0]) = \lim_{N\rightarrow +\infty} \epsilon_2(t,\mu[t,Z_N^0]) = 0$.

Historically, \cite{Dobrushin1979} introduced this methodology to obtain uniqueness results on solutions of Vlasov equations. In this article, the studied interaction functions are globally Lipschitz continuous, and the author does not resort to Grönwall lemma. With the same assumptions, a proof of Dobrushin stability was suggested by \cite{Golse2013}, theorem 1.4.3, making clear use of Grönwall lemma. In \cite{Lagoutiere2017}, the proposition 1 gives a quite similar contraction estimate, in the case where the transition function is expressed as a convolution product with the mean-field limit measure. In all aforementionned works, the transport functions $\G_N$ at microscopic level have the same expression as the transport function $\G$ at macroscopic level, and the physical models do not recquire to exclude the interaction of a particle with itself, notably thanks to a property of anti-symmetry of the underlying potential. In our case, the transition function is only assumed to be locally Lipschitz continuous, but this difficulty is bypassed by assuming that the i-nitial distribution $\mu_0$ has a compact support. The obtained upper-bound of $W_1\left(\mu[t,Z_N^0],\mu[t]\right)$ in (\ref{ineq_dobrushin}) is a much faster increasing function than in \cite{Golse2013}. The assumption on global Lipschitz continuity of the function $g$ leads to a factor of order $e^{K t}$ for some constant $K$, whereas the assumptions on quadratic variations of the functions, namely (A2) and (A4), leads to a factor of order $\exp\left(e^{Kt}\right)$ for some constant $K$, because of two subsequent applications of Grönwall lemma (see the proof in appendix \ref{proof_thm_dobrushin}). Needless to say that the upper bound in (\ref{ineq_dobrushin}) seems far from being optimal.

The next corollary uses the argument of Dobrushin stability to show the relation between the solution of the microscopic system (\ref{system_micro}) and the MFL characteristic flow.

\begin{corollary}
  \label{corollary_dobrushin}
  With the same assumptions as in theorem \ref{thm_dobrushin}, we consider the sequence of random variables $(z_n^0)_{n\in \N^*}$ independent and of distribution $\mu_0$. For all $N>1$, we define $Z_N^0 = (z_1^0,z_2^0,...,z_N^0)\in \Z^N$ the initial configuration of the system (\ref{system_micro}) and $t\in\R_+\mapsto (X_1(t,Z_N^0),...,X_N(t,Z_N^0))$ the solution of the system (\ref{system_micro}). Then we have
  \begin{equation*}
    X_1(t,Z_N^0)\conv{N}{+\infty}{a.s.} X_\infty(t,z_1^0)
  \end{equation*}
\end{corollary}

The above results provides a more visual intuition of the asymptotic link between the microscopic level of system (\ref{system_micro}) and the mean-field limit. The trajectories obtained by solving system (\ref{system_micro}) are more and more alike the trajectories given by the MFL characteristic flow $X_\infty$. A generalization to any sub-group of fixed size within the population can also be obtained. Indeed, for $k\in \N^*$ and for $N > k$, any sub-group of size $k$ $(X_{i_1}(t,Z_N^0),...,X_{i_k}(t,Z_N^0))$, with $i_1,...,i_k$ being distinct integers in $\llbracket 1;N\rrbracket$, has the same distribution as $(X_1(t,Z_N^0),...,X_k(t,Z_N^0))$ by symmetry. According to the previous corollary, the almost sure convergence for a single individual can be generalized to any sub-group of size $k$.
\begin{equation*}
  (X_1(t,Z_N^0),...,X_k(t,Z_N^0))\conv{N}{\infty}{a.s.}(X_\infty(t,z_1^0),...,X_\infty(t,z_k^0))
\end{equation*}
The limit distribution of the sequence of random variables $((X_1(t,Z_N^0),...,X_k(t,Z_N^0))_{N > k}$ is factorized and is exactly $\mu[t]^{\otimes k}$, as $(X_\infty(t,z_1^0),...,X_\infty(t,z_k^0))\sim \mu[t]^{\otimes k}$. For finite $N$ and for $t >0$, the random variables $X_1(t,Z_N^0),...,X_k(t,Z_N^0)$ are strongly interdependent. At the limit $N\rightarrow +\infty$, the individuals are independent. More accurately, if one focuses on a finite group of individuals, while the rest of the population is increasing towards infinity, then these observed individuals have independent trajectories in the probabilistic sense. Their distribution is said to be asymptotically factorized. An alternative proof of the phenomenon of chaos propagation is given in \cite{Golse2013b}, section 1.6. This proof is based on a characterization of asymptotically factorized sequence of probability measures (see theorem 1.6.2 in \cite{Golse2013b}).

The phenomenon of chaos propagation may have applications for statistical inference, paving the way for methodologies based on variational Bayes approximation. Let us consider the following example : we aim at studying the dynamics of an heterogeneous crop from the observation of the growth of few dozens of plants. Their growth is assumed to be well represented by a model of the form of \cite{Schneider2006}, but some parameters of the interaction function $g$ are unknown. In general, we do not know accurately the exact number $N$ of individuals in the population, but we know that $N$ is much larger than the number of observed individuals. In a Bayesian setting, i.e. when we want to compare prior knowledge and assumptions with field observations, the resulting inference problem is of great difficulty. Among other things, it requires to determine the posterior distribution of the number of individuals in the population\footnote{A possible prior for the random variable $N$ would be a Poisson distribution.}, but also the posterior distributions of all the unobserved individuals, i.e. of their positions and of their characteristics $\gamma$ and $S$. This is clearly intractable for a population having the dimension of a crop. Otherwise, if we make the approximation that the observed individuals are in interaction with an infinity of individuals, which is quite a relevant approximation after all, and that this continuum of individuals is represented by the MFL distribution $\mu[t]$, then the inference problem is significantly simplified : the observed individuals are then mutually independent, and there is no need to extract the information of all the unobserved individuals. Of course, the difficulty is elsewhere : how to simulate the MFL distribution efficiently, so that it can be used within a statistical inference process. The next section gives a first attempt to answer this issue.

\section{Simulation of the MFL distribution using Gaussian process regression}
\label{section_simulation}

In this section, we present a preliminary work on the numerical approximation of the MFL distribution $t\mapsto \mu[t]$, which is defined as the measure-solution of variational problem (\ref{weak_transport_macro}). So it boils down to solve numerically a hyperbolic PDE with non-local velocity. The simulation of solutions of kinetic equations is a well-documented in the literature. Amongst others, we can quote the upwind scheme introduced by \cite{Lagoutiere2017}, which consists in a reconstruction of the solution using finite volumes. The reconstruction is piecewise constant over a discretization of the phase space $\Z$. In the case of Schneider heterogeneous population model, the space is of dimension higher than 3, and this makes the discretization of the space a too expensive task on the computational view point. This constraint of the dimension leads rather towards mesh-free methods.

The method suggested here consists in approximating by regression a consistent sequence of reconstructions of the exact characteristic flow. It is therefore a semi-Lagrangian method with an interpolation step. The family of functions used for the interpolation is defined from the interaction function $g$, and takes the form of linear combinations of reproducing kernels. The proof of the consistency of the scheme is an on-going work. However, some numerical tests seems to confirm that this approach is relevant.

For the sake of simplicity, the method is presented through the simulation of the Schneider model. In this case, the MFL distribution is the law of the random variable $(r_\infty(t,r^0,\theta),\theta)$, where $r^0\sim \delta_{r^0}$ is a constant, where $\theta\sim p_0^\theta$ is defined in equation (\ref{eq_mu0_schneider}), and $r_\infty$ is the characteristic flow defined by equation (\ref{eq_flow_r}). By change of variable, we can consider the characteristic flow associated to the size variable $s$, which is defined as the solution of the equation
\begin{equation}
  \begin{aligned}
    &\forall \theta \in \Theta,~\left\{\begin{array}{l}
    s_\infty(0,s^0,\theta) = s^0\\
    \displaystyle \frac{\partial s_\infty}{\partial t}(t,s^0,\theta) = \int_\Theta g(s_\infty(t,s^0,\theta),\theta,s_\infty(t,s^0,\theta'),\theta')p_0(\theta')\lambda^{\otimes 4}(\intd \theta')
    \end{array}\right.\\
    &\text{with }g(s,(\vec{x},S,\gamma),s',(\vec{x}',S',\gamma')) = \gamma s\left(\log(S/s_m)(1-C(s,s',|\vec x-\vec x'|))-\log(s/s_m)\right)
  \end{aligned}
  \label{eq_flow_s}
\end{equation}
So our aim is to approximate the function $(t,\theta)\in \R_+\times \Theta \mapsto s_\infty(t,s^0,\theta)$.

A direct resolution of equation (\ref{eq_flow_s}) using an explicit Euler method, with time discretization $\Delta t >0$, chosen small enough, would lead to a sequence of functions $(s_n)_{n\in \N}$ defined an induction equation.
\begin{equation}
  \forall \theta\in \Theta,~ \left\{\begin{array}{l}
  s_0(\theta) = s^0\\
  \displaystyle \forall n\in \N,~s_{n+1}(\theta) = s_n(\theta) + \Delta t\int_\Theta g(s_n(\theta),\theta,s_n(\theta'),\theta')p_0^\theta(\theta')\lambda^{\otimes 4}(\intd \theta')
  \end{array}\right.
\end{equation}
This sequence of functions cannot be computed exactly, as the integral is not analytical. This integral is in fact an expectation with respect to the density $p_0^\theta$. Let $\omega_M = (\theta_i^\omega)_{1\leq i\leq M}$ be sample of the distribution $\mu_0^\theta$ of density $p_0^\theta$. We consider the sequence of functions $(s_n(.,\omega_M))$ defined as the empirical approximation of the sequence $(s_n)_{n\in \N}$ using the sample $\omega_M$.
\begin{equation}
  \forall \theta\in \Theta,~\left\{\begin{array}{l}
  s_0(\theta,\omega_M) = s^0\\
  \displaystyle \forall n\in \N,~ s_{n+1}(\theta,\omega_M) = s_n(\theta,\omega_M) + \frac{\Delta t}{M}\sum_{i=1}^Mg(s_n(\theta,\omega_M),\theta,s_n(\theta^\omega_i,\omega_M),\theta_i^\omega)
  \end{array}\right.
  \label{eq_induction_r}
\end{equation}
It is quite straightforward to prove that for any fixed $n\in \N$, the sequence $(s_n(.,\omega_M))_{M\in \N^*}$ is an almost sure approximation of the characteristic flow at time $n\Delta t$.
\begin{equation*}
  \forall \theta\in \Theta,~\forall n\in \N,~s_n(\theta,\omega_M) \conv{M}{\infty}{a.s.}s_n(\theta)
\end{equation*}
Indeed, the sequence of functions $(s_n(.,\omega_M))$ is stochastic because of its dependency with respect to the sample $\omega_M$. The above convergence is mainly based on the law of large numbers, enabling to prove a uniform almost sure convergence over the space $\Theta$. So $(s_n(.,\omega_M))_{n\in \N}$ constitutes a simple approximation of the characteristic flow, but it has some limitations. It can only give a local estimation of the function $s_n$. Indeed, to compute $s_{100}(\theta,\omega_M)$ at a given point $\theta_0\in \Theta$, then it requires the computation of $s_{99}(\theta_0,\omega_M)$, and in turn the computation of $s_{98}(\theta,\omega_M)$, etc... We cannot know the values of the function $s_n(.,\omega_M)$ outside of the set of points we have decided to observe a priori, from the very initial time $n=0$, this set of observation points including also the sample $\omega_M$. For a global approximation of the function $s_n$, a grid covering the whole space $\theta$ has to be build, so this boils down exactly to the construction of a mesh, which is to be avoided in our case. An interpolation method is used at this point so that the local information given by some values of $s_n(.,\omega_M)$ could be extended to the whole space $\Theta$.

The basis of functions used for interpolation has been chosen from a qualitative estimation of the correlation between the values of the function $s_n(.,\omega_M)$. According to central limit theorem, the asymptotic covariance matrix of the random variables $s_1(\theta_1,\omega_M)$ and $s_1(\theta_2,\omega_M)$ for $\theta_1$ and $\theta_2$ in $\Theta$ when $M\rightarrow +\infty$ depends on the interaction $g$.
\begin{equation*}
  \begin{aligned}
    &\sqrt{M}\begin{pmatrix} s_1(\theta_1,\omega_M)-s_1(\theta_1)\\ s_1(\theta_2,\omega_M)-s_1(\theta_2)\end{pmatrix} \conv{M}{\infty}{\mathcal{L}}\mathcal{N}_2\left(0,\Delta t^2\Sigma_0(\theta_1,\theta_2)\right)\\
    &\text{with }\Sigma_0(\theta_1,\theta_2) = \begin{pmatrix} \Cov_0(\theta_1,\theta_1)& \Cov_0(\theta_1,\theta_2)\\
      \Cov_0(\theta_1,\theta_2)&\Cov_0(\theta_2,\theta_2)\end{pmatrix}\\
    &\Cov_0(\theta_1,\theta_2) = \Cov_{\theta'}(g(s^0,\theta_1,s^0,\theta'),g(s^0,\theta_2,s^0,\theta'))\\
  \end{aligned}
\end{equation*}
$\forall d\in \N^*,$ $\mathcal{N}_d(\mu,\Sigma)$ is the normal distribution of mean $\mu$ and of covariance matrix $\Sigma$. In other words, the random vector $(s_1(\theta_1,\omega_M),s_1(\theta_2,\omega_M))^\T$ behaves approximately like a Gaussian vector. From this result, we make the approximation that this property holds for all $n\in \N^*$\footnote{this approximation may not be justified theoretically.}.
\begin{equation*}
  \begin{aligned}
    &\forall \theta_1,\theta_2\in \Theta, ~\forall n\in \N,~\begin{pmatrix} s_n(\theta_1,\omega_M)\\s_n(\theta_2,\omega_M)\end{pmatrix}\sim \mathcal{N}_2\left(\begin{pmatrix} s_n(\theta_1)\\s_n(\theta_2)\end{pmatrix},\frac{\Delta t^2}{M}\Sigma_{n-1}(\theta_1,\theta_2)\right)\\
    &\Sigma_{n-1}(\theta_1,\theta_2) = \begin{pmatrix} \Cov_{n-1}(\theta_1,\theta_1)& \Cov_{n-1}(\theta_1,\theta_2)\\
          \Cov_{n-1}(\theta_1,\theta_2)&\Cov_{n-1}(\theta_2,\theta_2)\end{pmatrix}\\
    &\Cov_{n-1}(\theta_1,\theta_2) = \Cov_{\theta'}(g(s_{n-1}(\theta_1),\theta_1,s_{n-1}(\theta'),\theta'),g(s_{n-1}(\theta_2),\theta_2,s_{n-1}(\theta'),\theta'))
    \end{aligned}
\end{equation*}
This reasonable expression of the covariance leads to a choice of interpolation functions being defined from the covariance function, which is by construction a positive kernel.
\begin{equation*}
  k_n(\theta_1,\theta_2) = \frac{\Delta t^2}{M}\Cov_{\theta'}(g(s_{n-1}(\theta_1),\theta_1,s_{n-1}(\theta'),\theta'),g(s_{n-1}(\theta_2),\theta_2,s_{n-1}(\theta'),\theta'))
\end{equation*}
This kernel cannot be used \textit{per se} as the sequence $(s_n)_{n\in \N}$ is unknown. Another kernel is therefore chosen, but still largely inspired from the above expression. As the sequence $(s_n)_{n\in \N}$, they are replaced in the above expression by parametric functions that reproduce roughly their variations over the space $\Theta$. More specifically, polynomial functions $(m^s_n)_{n\in \N}$ of degree 2 were chosen to approximate the sequence $(s_n)_{n\in \N}$.
\begin{equation*}
  m^s_n(\theta) = m^r(\theta;a_n,b_n,c_n) = a_n+b_n^\T\theta+c_n^\T v_{16}(\theta\theta^\T)
\end{equation*}
where $v_{16}:\mathcal{M}_4(\R)\rightarrow \R^{16}$ is the canonical bijection between the square matrices $4\times 4$ and the vector of 16 components. The coefficients $(a_n,b_n,c_n)$ are chosen so that the parametric function is close to function $s_n$ is the $L^2$ sense.
\begin{equation*}
  (a_n,b_n,c_n) = \underset{a,b,c}{\text{argmin}}~\int_\Theta (s_n(\theta)-m^r(\theta;a,b,c))^2p_0^\theta(\theta)\lambda^{\otimes 4}(\intd \theta)
\end{equation*}
Equivalently, $(a_n,b_n,c_n)$ is the solution of a linear system expressed with expectations with respect to the density $p_0^\theta$.
\begin{equation}
  \expect\begin{pmatrix}
  1&\theta^\T&v_{16}(\theta\theta^\T)^\T\\
  \theta& \theta\theta^\T&\theta v_{16}(\theta\theta^\T)^\T\\
  v_{16}(\theta\theta^\T)&v_{16}(\theta\theta^\T)\theta^\T&v_{16}(\theta\theta^\T)v_{16}(\theta\theta^\T)^\T
  \end{pmatrix} \begin{pmatrix} a_n\\b_n\\c_n\end{pmatrix} = \expect\begin{pmatrix} s_n(\theta)\\s_n(\theta)\theta\\s_n(\theta)v_{16}(\theta\theta^\T)\end{pmatrix}
    \label{linear_abc}
\end{equation}
In this linear system, the functions $s_n$ can be replaced by their stochastic approximations $s_n(.,\omega_M)$, and the theoretical mean can be replaced by an empirical mean over the set $\omega_M$.

The final expression for the kernel used for the interpolation depends on the sample $\omega_M$.
\begin{equation*}
  \begin{aligned}
    &k_n(\theta_1,\theta_2) = \frac{\Delta t^2}{M}\left(\frac{1}{M}\sum_{i=1}^Mg(m^s_{n-1}(\theta_1),\theta_1,s_{n-1}(\theta_i^\omega),\theta_i^\omega)g(m^s_{n-1}(\theta_2),\theta_2,s_{n-1}(\theta_i^\omega),\theta_i^\omega)\right.\\
    &\left.- \frac{1}{M^2}\left(\sum_{i=1}^Mg(m^s_{n-1}(\theta_1),\theta_1,s_{n-1}(\theta_i^\omega),\theta_i^\omega)\right)\left(\sum_{i=1}^Mg(m^s_{n-1}(\theta_2),\theta_2,s_{n-1}(\theta_i^\omega),\theta_i^\omega)\right)\right)
  \end{aligned}
\end{equation*}
The theoretical covariance is replaced by an empirical covariance over the sample $\omega_M$ and the characteristic flows $s_{n-1}$ are replaced by either the polynomial functions $m^s_{n-1}$ either the stochastic approximation $s_{n-1}(.,\omega_M)$. If the parametric model $m^s_n$ is not too rough and if $M$ is large, then the above covariance function $k_n$ is consistent with the stochastic behaviour of $s_n(.,\omega_M)$, and $k_n$ is easy to evaluate over the whole space.

In addition to the values $s_n(\omega_M,\omega_M) = (s_n(\theta_i^\omega,\omega_M))_{1\leq i\leq M}$, the function $s_n(.,\omega_M)$ is evaluated over another set of points $\Theta_{1:K} = (\theta_j)_{1\leq j\leq K}$, called \textit{training set}, that can also be taken as a sample from the density $p_0^\theta$. For all $n\in \N$, we extend the values of $s_n(\Theta_{1:K},\omega_M)$ by making the approximation that the values of $s_n(.,\omega_M)$ is a Gaussian process of mean function $\theta\mapsto m^s_n(\theta)$ and of covariance function $(\theta_1,\theta_2)\in \Theta^2\mapsto k_n(\theta_1,\theta_2)$ (cf. \cite{Rasmussen2004} for an introduction to Gaussian processes). In particular, under this approximation, for all $\theta\in \Theta$
\begin{equation*}
  \begin{pmatrix} s_n(\theta,\omega_M)\\ s_n(\Theta_{1:K},\omega_M)\end{pmatrix} \sim \mathcal{N}_{K+1}\left(\begin{pmatrix}m^s_n(\theta)\\m^s_n(\Theta_{1:K})\end{pmatrix},\begin{pmatrix} k_n(\theta,\theta)&k_n(\theta,\Theta_{1:K})\\ k_n(\Theta_{1:K},\theta)&k_n(\Theta_{1:K},\Theta_{1:K})\end{pmatrix}\right)
\end{equation*}
The distribution of $s_n(\theta,\omega_M)$ is given by conditioning with respect to the observed, or rather computed, values of $s_n(\Theta_{1:K},\omega_M)$.
\begin{equation*}
  \begin{aligned}
    &s_n(\theta,\omega_M)|s_n(\Theta_{1:K},\omega_M) \sim \mathcal{N}_1\left(m^s_n(\theta)+k_n(\theta,\Theta_{1:K})k_n(\Theta_{1:K},\Theta_{1:K})^{-1}(s_n(\Theta_{1:K},\omega_M)\right.\\
    &\left.-m^s_n(\Theta_{1:K})),k_n(\theta,\theta)-k_n(\theta,\Theta_{1:K})k_n(\Theta_{1:K},\Theta_{1:K})^{-1}k_n(\Theta_{1:K},\theta)\right)
  \end{aligned}
\end{equation*}
Therefore, under this approximation, the most probable value of $s_n(\theta,\omega_M)$ knowing the values of $s_n(\Theta_{1:K},\omega_M)$ is given by the mode of the above conditional distribution. This is the reconstruction of the characteristic flow we use to estimate it over the whole space $\Theta$.
\begin{equation}
  \begin{aligned}
    &\hat s_n(\theta,\omega_M) = m^s_n(\theta) + \sum_{j=1}^K\alpha_{j,n}k_n(\theta,\theta_j)\\
    &\text{with } \alpha_n = k_n(\Theta_{1:K},\Theta_{1:K})^{-1}(s_n(\Theta_{1:K},\omega_M)-m^s_n(\Theta_{1:K}))
  \end{aligned}
\end{equation}
One can notice that more information could have been used to compute the conditional distribution, as we also know the values of $s_n(\omega_M,\omega_M)$. The reason why the sample $\omega_M$ is omitted is just that inverting a matrix of dimension $K$ is cheaper than inverting a matrix of dimension $M+K$.

The relevancy of the reconstruction can be assessed using a test set $\Theta^t_{1:K} = (\theta^t_j)_{1\leq j\leq K}$, that can also be a sample drawn from density $p_0^\theta$. A mean square error is used for this purpose.
\begin{equation*}
  \forall n\in \N^*,~J_n = \sqrt{\frac{1}{K}\sum_{j=1}^K(\hat s_n(\theta^t_j,\omega_M)-s_n(\theta^t_j,\omega_M))^2}
\end{equation*}
If $J_n$ remains relatively small during the iterations, then the reconstruction of $s_n(.,\omega_M)$ is likely to be relevant.

The different steps of the simulation process are summarized in the following algorithm.
\begin{algorithm}
  \caption{Approximation of the characteristic flow of Schneider model}
  \begin{algorithmic}
    \State \textbf{Input} : size $M$ of the sample $\omega_M$, size $K$ of the training set $\Theta_{1:K}$ and of the testing set $\Theta^t_{1:K}$, $n_{\max}$ the maximal number of iterations.
    \State \textbf{Initialization} :
    \begin{enumerate}
    \item draw a sample $\omega_M$, a training set $\Theta_{1:K}$ and a testing set $\Theta^t_{1:K}$ from the density $p_0^\theta$.
    \item initialization of the characteristic flow $s_0(\Theta_{1:K},\omega_M) = s_0(\Theta^t_{1:K},\omega_M) = r^0\mathbbm{1}_K$, $s_0(\omega_M,\omega_M) = r^0\mathbbm{1}_M$, and the parameters of the $m^r$ functions : $a_0 = r^0$, $b_0 = 0$ and $c_0 = 0$.
    \end{enumerate}
    \For{$n=1:n_{\max}$}
    \begin{enumerate}
    \item update the characteristic flow over the sets $\omega_M$, $\Theta_{1:K}$ and $\Theta^t_{1:K}$ using the induction equation (\ref{eq_induction_r}).
    \item compute the coefficients of $m_n^r$ using the sample $\omega_M$, $s_n(\omega_M,\omega_M)$ to approximate the coefficients of linear system (\ref{linear_abc}).
    \item compute the coefficients $\alpha_n$ by solving the linear system $k_n(\Theta_{1:K},\Theta_{1:K})\alpha_n = s_n(\Theta_{1:K},\omega_M)-m^s_n(\Theta_{1:K})$.
    \item compute the test error $J_n$
    \end{enumerate}
    \EndFor
    \State \textbf{Output} : $(\alpha_n)_{1\leq n\leq n_{\max}}$, $(a_n,b_n,c_n)_{1\leq n\leq n_{\max}}$, $(J_n)_{1\leq n\leq n_{\max}}$.
  \end{algorithmic}
\end{algorithm}

The algorithm was run with the parameters values given in table \ref{table_schneider} and for $n_{\max} = 100$ iterations, with a sample size $M = 1000$ and size of training / reconstruction set of $K = 100$. Figure \ref{test_error} displays the evolution of the test error of the reconstruction $\hat s_n(.,\omega_M)$, along with the test error associated with the polynomial approximation $m^s_n$.
\begin{figure}[H]
  \begin{center}
    \includegraphics[width = \textwidth]{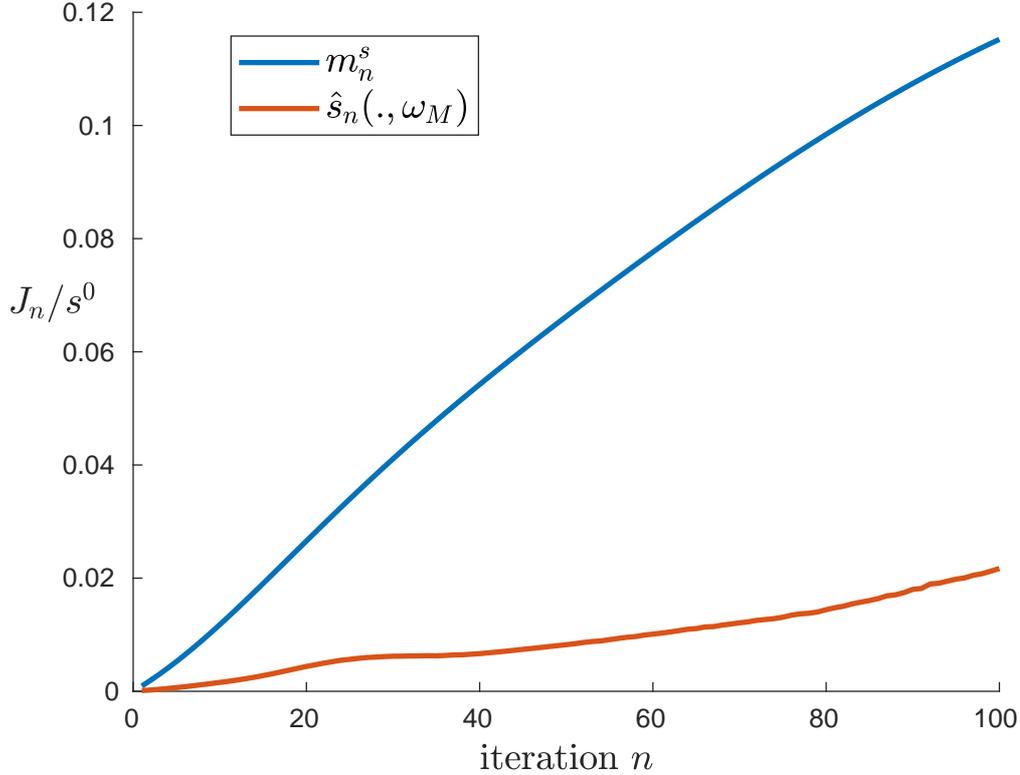}
  \end{center}
  \caption{Evolution of the test errors, renormalized by the initial size $s^0$, measuring the discrepancy between $s_n(.,\omega_M)$ and $m^s_n$ (blue curve) $J_n^p = \sqrt{\frac{1}{K}\sum_{j=1}^K(s_n(\theta^t_j,\omega_M)-m_n^s(\theta_j^t))^2}$, and between $s_n(.,\omega_M)$ and $\hat s_n(.,\omega_M)$ (red curve) $J_n = \sqrt{\frac{1}{K}\sum_{j=1}^K(s_n(\theta^t_j,\omega_M)-\hat s_n(\theta^t_j,\omega_M))^2}$}
  \label{test_error}
\end{figure}

Figure \ref{test_error} shows that both the polynomial approximation and Gaussian process (GP) reconstruction seem to provide a good estimate of the function $\hat s_n(.,\omega_M)$, with a relative remaining lower than $12\%$ for the polynomial approximation, and lower than $2\%$ for the GP reconstruction. The error made by the polynomial function increases almost linearly with the iterations, meaning that the shapes of the functions $(\hat s_n(.,\omega_M))_{n\in \N^*}$ become more and more complex for large $n$, and the approximation by parabolic functions become more and more rough. As a matter of fact, the GP reconstruction has also an increasing test error, but it still provides a significant improvement with respect to the polynomial approximation.

Once $\hat s_n(.,\omega_M)$ is computed, an approximate sample of the marginal distribution of random variable $s^n\sim\mu^s[n\Delta t]$ can be drawn. The sample is obtained by drawing independent samples $(\theta_i')_{1\leq i\leq M}$ from density $p_0^\theta$ and compute the values of the characteristic flow over this sample $(\hat s_n(\theta_i',\omega_M))_{1\leq i\leq M}$. For $n >0$, the marginal distribution $\mu^s[n\Delta t]$ is absolutely continuous with respect to the Lebesgue measure $\lambda^{\otimes 4}$, and the associated density can be estimated by non-parametric kernel regression. We define $p_n^s:s\in \R\mapsto \displaystyle \frac{\partial \mu^s[n\Delta t]}{\partial \lambda^{\otimes 4}}(s)\in \R_+$ the associated density. Figure \ref{densities} illustrates the evolution of the marginal density of $s^n$ with the time.

\begin{figure}[H]
  \begin{center}
    \includegraphics[width = \textwidth]{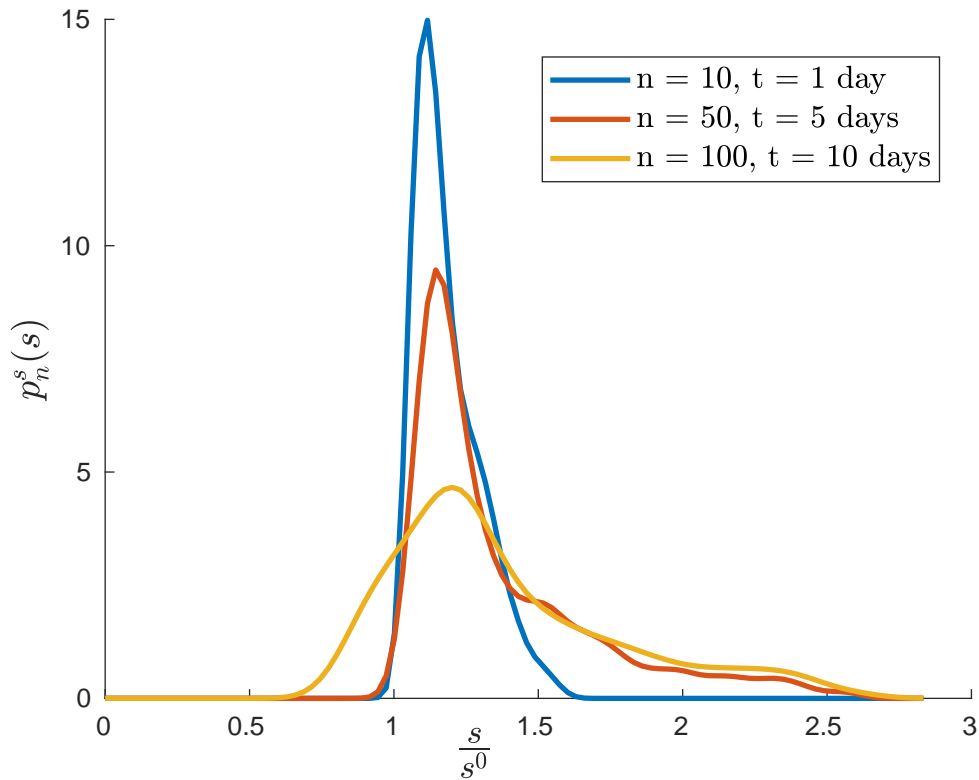}
  \end{center}
  \caption{Evolution of the marginal density of random variable $s^n\sim \mu^s[n\Delta t]$ estimated using Beta kernels from a sample $(\hat s_n(\theta'_i,\omega_M))_{1\leq i\leq M}$ for $n = 10$ ($t = 1$ day), $n = 50$ ($t = 5$ days) and $n = 100$ ($t = 10$ days).}
  \label{densities}
\end{figure}

In figure \ref{densities}, we can observe the distribution of the sizes is in the beginning above the initial size $s^0$ : this is the first stage of the growth in the population, when all plants have their sizes increasing. This corresponds to the densities at times $t = 1$ day and $t = 5$ days. At some point, the competition becomes too important, mainly at the center of the domain $[0;L]^2$ and part of the plants decay, leading to the appearance of plant with sizes lower than $s^0$ at time $t = 10$ days. Besides, the plants that keep on increasing are the ones that are located in close to the edge $x = L, y = L$, which have faster growth rates $\gamma$ and taller asymptotic sizes $S$. These plants therefore their equilibrium size faster than in the rest domain, so that there are very little change between density at $t = 5$ days and density at time $t = 10$ days for the plants of size higher than $1.5~s^0$. This result is consistent with the simulations of the differential system (\ref{system_schneider}) displayed in figure \ref{simu_schneider1}.

A clearer visualization of the global behaviour of the MFL distribution can be made by computing the surface corresponding to the averaged size over the domain $[0;L]^2$, i.e. the expectation $(x,y)\in [0;L]^2\mapsto \hat e_n(x,y) = \expect_{\mu[n\Delta t]}(\hat s_n(\theta,\omega_M)|x,y)$, which is obtained by marginalizing growth parameters $\gamma$ and $S$.
\begin{equation*}
  \begin{aligned}
    &\expect_{\mu[n\Delta t]}(\hat s_n(\theta,\omega_M)|x,y) = \frac{1}{\sigma_S\sigma_\gamma}\int_{S_1(x)}^{S_2(x)}\int_{\gamma_1(y)}^{\gamma_2(y)}\hat s_n((x,y,S,\gamma);\omega_M)\intd \gamma\intd S
  \end{aligned}
\end{equation*}
\begin{equation*}
  \begin{aligned}
    &\hat e_n(x,y) \approx \frac{1}{M}\sum_{i=1}^M \hat s_n((x,y,S_1(x) + \sigma_S u_i, \gamma_1(y) + \sigma_\gamma u'_i);\omega_M)
  \end{aligned}
\end{equation*}
where $(u_i)_{1\leq i\leq M}$ and $(u'_i)_{1\leq i\leq M}$ are independent samples from the uniform distribution $\U([0;1])$.

\begin{figure}[ht!]
  \begin{center}
    \subfloat[$t = 1$ day, $n = 10$]{
      \includegraphics[width = 0.5\textwidth]{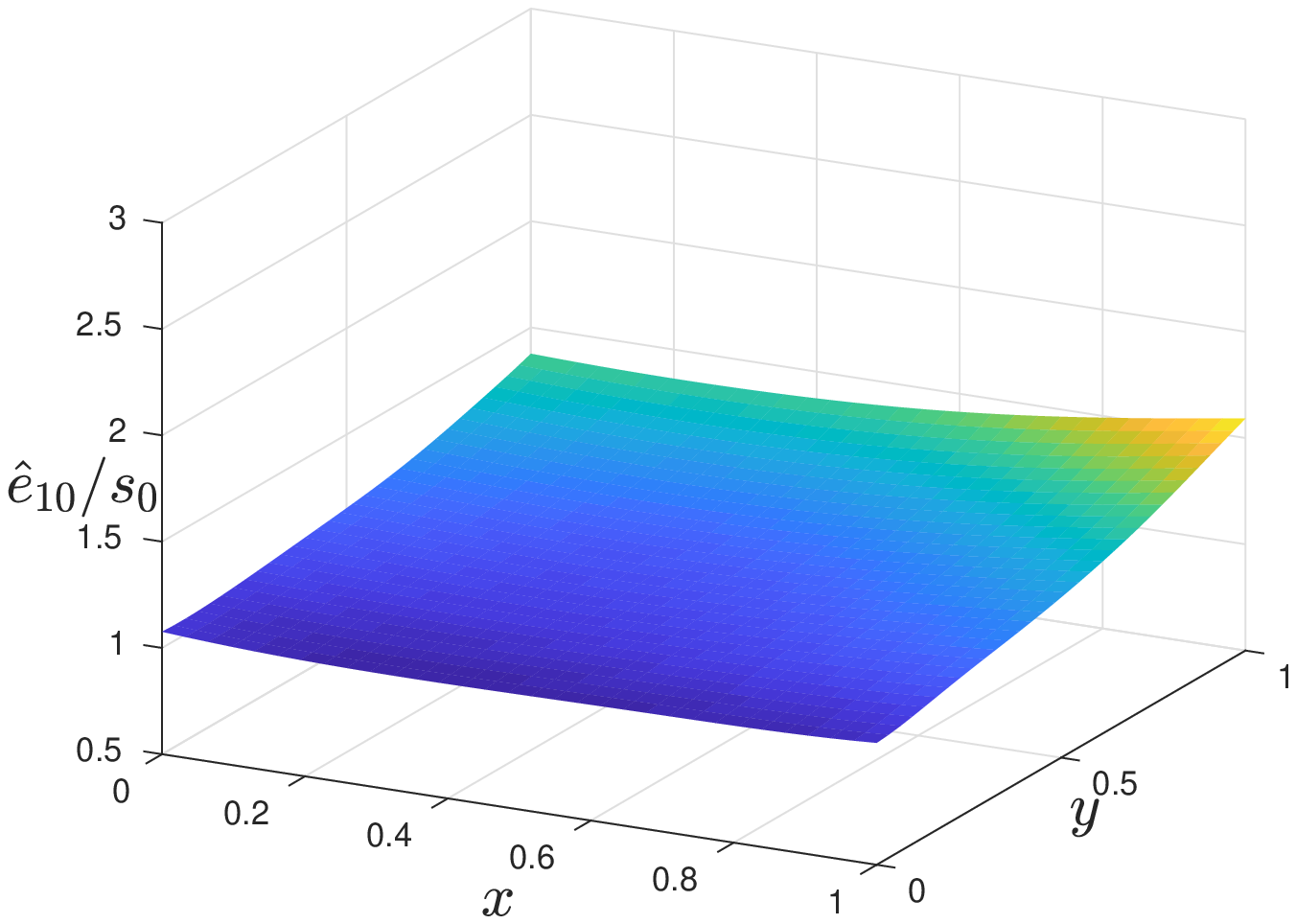}}
    \subfloat[$t = 5$ days, $n = 50$]{
      \includegraphics[width = 0.5\textwidth]{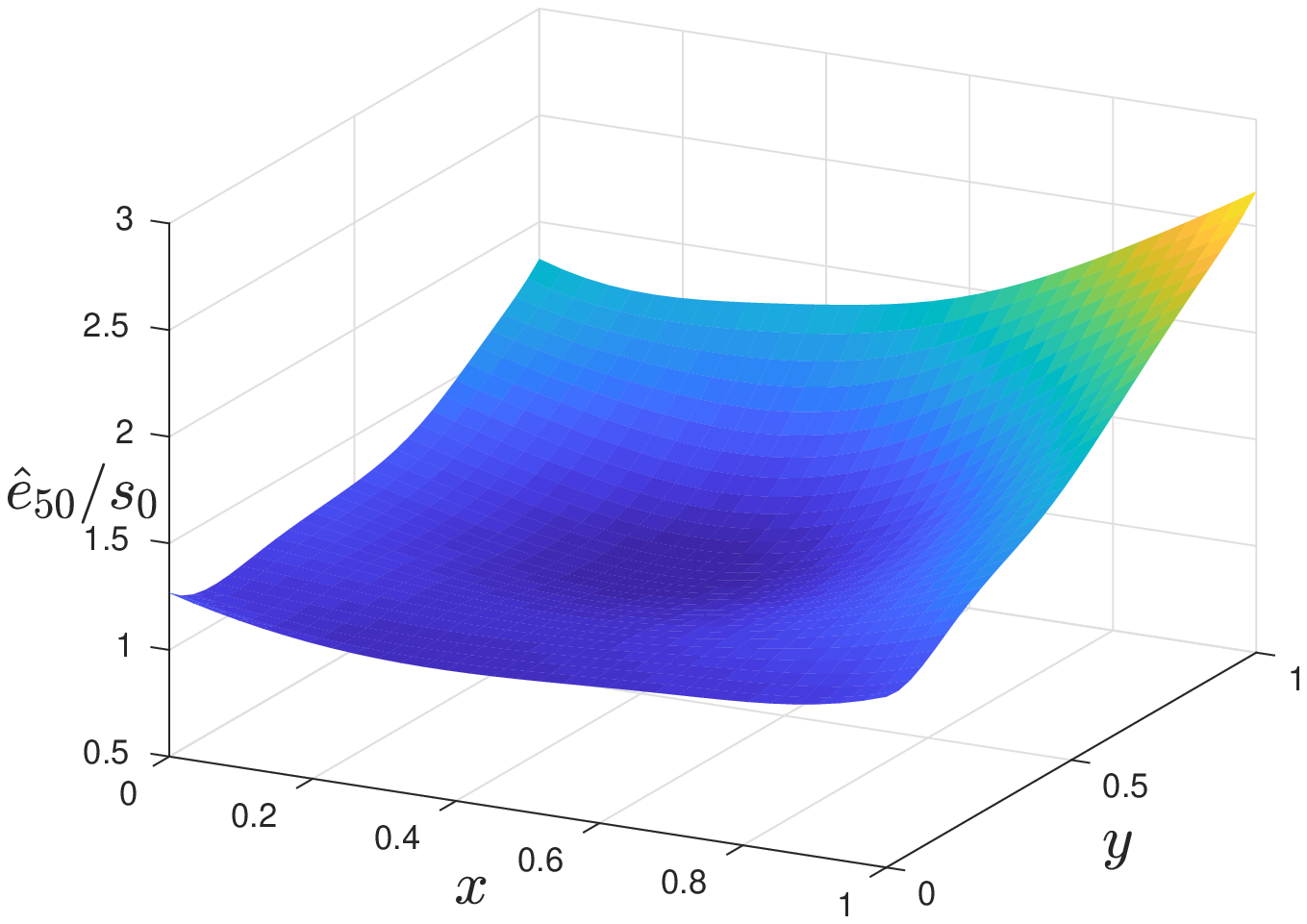}}\\
    \subfloat[$t = 10$ days, $n = 100$]{
      \includegraphics[width = 0.5\textwidth]{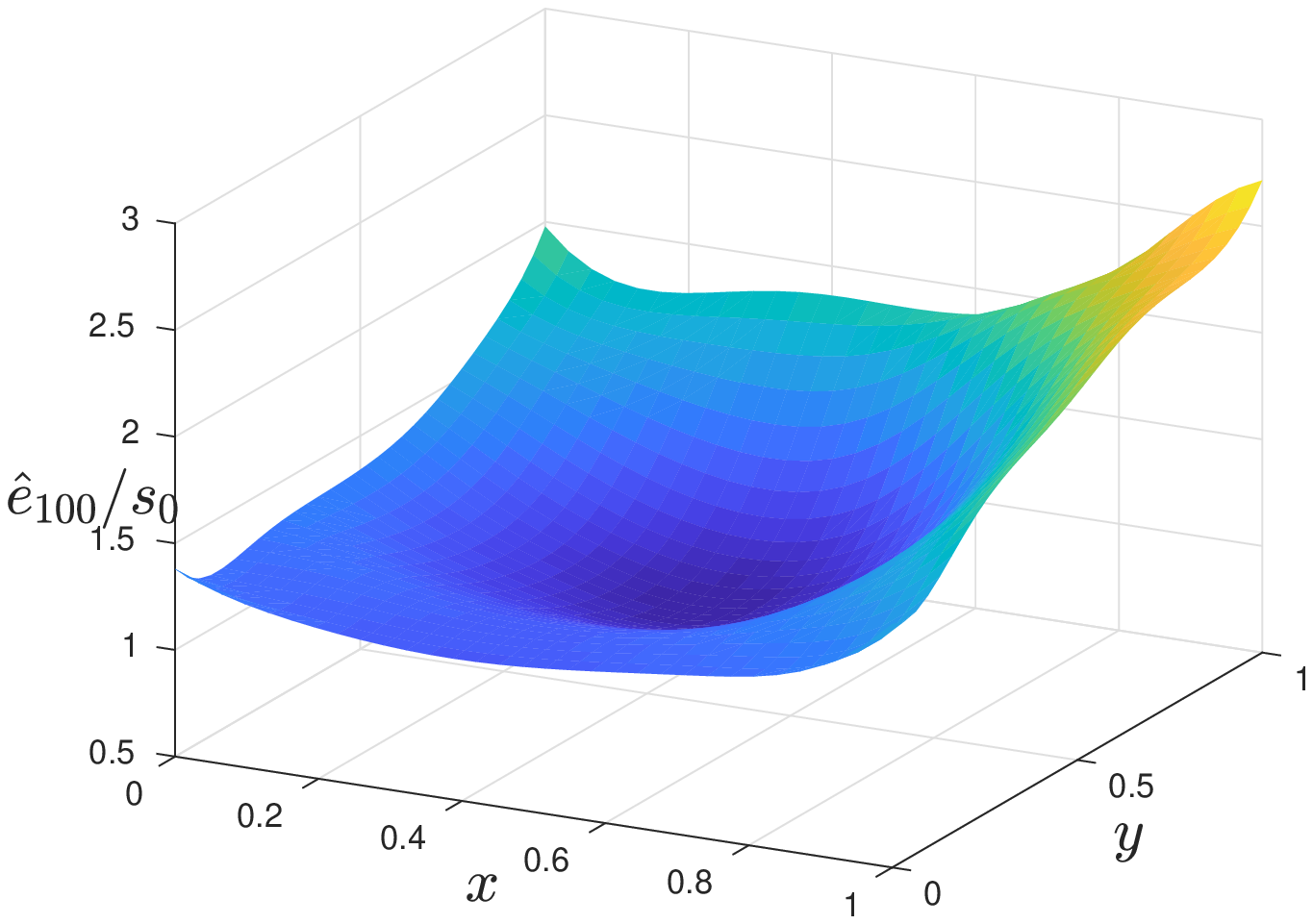}}
    \caption{Evolution of the surface $(x,y)\mapsto \hat e_n(x,y)$ with the iterations}
  \end{center}
\end{figure}
As expected, the surface has its maximal value at the point $x = L, y = L$.The line $x\mapsto \hat e_n(x,L)$ does not change much from $n=50$ to $n = 100$, because plants along this line are already close to their equilibrium size (with competition) at $n = 50$, whereas the line $x\mapsto \hat e_n(x,0)$ has not converged yet for $n = 50$, due to its slower growth rate. As $n$ tends towards infinity, we can expect the surface to be more and more invariant by translation along $y-$axis.

\section{Conclusion and perspectives}

Heterogeneous population models can be approximated by the MFL distribution when the population is large enough and when the interaction function describing the dynamics satisfies a set of assumptions. The phenomenon of chaos propagation, implied by Dobrushin stability, seems to provide an interesting research direction to circumvent the problem of fully-correlated individuals, that arises when the inference of the model is carried out. The suggested methodology for the simulation of the MFL distribution gives promising results, although a theoretical analysis of its consistency still needs to be conducted (on-going work). Our next step is to apply the MFL approximation to real experimental data, to study the impact of competition on the development of plants in mixed stands.

The MFL distribution is appealing because of the property of chaos propagation. But it is clear that MFL approximation might not be relevant for populations having relatively small sizes, as it can be expected when looking at figure \ref{simu_schneider1}. A limit seems to be reached for $N > 100$, as there are very few changes in the dynamics above this threshold. For smaller $N$ however, the trajectory of a single individual is noisy, and the approximation of the population distribution by a factorized distribution might be too rough. In general terms, the critical size of the population $N_c$ is a function of the tolerance $\epsilon$ on some metric quantifying the discrepancy between the \textit{microscopic} distribution and the MFL distribution, of the length $T$ of the time interval during which the system is observed and finally, of course, of the transition function $g$\footnote{An obvious example is given by a transition function $g$ having no dependency with respect to $X',\theta'$. In that case, the individuals defined by system \ref{system_micro} are in fact independent already, and the critical size is then $N_c = 1$ for all $\epsilon,T$}. The metric, over which a tolerance $\epsilon$ is defined, has to be chosen according to the objectives of the inference. For instance, if our aim is to compute the posterior distribution of some parameter of the model given a set of observations, then we need to find an estimate of the discrepancy between the result that would be obtained by direct inference and the result obtained under MFL approximation. This task seems rather unfeasible, as both of these distributions are either too difficult to compute or simulated with a procedure having a yet uncontrolled error. The upper bound provided by Dobrushin stability in inequality (\ref{ineq_dobrushin}) is far too rough to be used for the estimation of the critical size of the population $N_c$.

For some systems however, MFL approximation is without doubt relevant. This is the case, for instance, for systems studied in statistical physics, systems that are constituted by a number of particles near or beyond the Avogadro constant ($\approx 6\times 10^{23}$). Even in this \textit{favourable case}, the use of MFL approximation within a process of Bayesian inference is not well set yet. It would require the coupling of a numerical scheme, similar to the one presented in the previous section, with a time-consuming MH or MHWG algorithms. In machine learning and signal processing literature (\cite{Marnissi2016}), the distribution of the variational Bayes approximation is mainly chosen for its conjugation property with the prior distribution. This often leads to analytical posterior distributions of the parameters, and it may spare a lot of computation time. In our case, our choice is motivated by the behaviour of the dynamical system when it becomes infinite. There is no chance that applying Bayes rule in this context might lead to known or tractable posterior distributions. The solution to this issue may consist in a trade-off between traditional variational Bayes techniques, that are efficient but biased by convenience-motivated choices, and the simulation of the MFL distribution associated to the system, which may require a significant amount of computation time but which is asymptotically unbiased.

This paper has focused on a quite specific class of interaction models, namely the ones that can be decomposed into a sum of pairwise potentials. In the case of more realistic plant models, such decomposition cannot be obtained. The competition is not considered as being additive, and does not even have a closed-form expression in some cases. The necessay complexity of these models leads to the question of the derivation of MFL distribution associated to more generic dynamical systems. In our case, the velocity field $\G_N$ at the microscopic level has a linear dependency with respect to the population empirical measure $\mu[0,Z_N^0]$. The convergence of $\G_N$ towards a MFL velocity field may still be preserved when $\G_N$ is only continuous with respect to $\mu[0,Z_N^0]$ for some Wasserstein metric. Such theoretical study, in a more general setting than the one presented in this paper, may enable to study the asymptotic behaviour of more realistic plant population models, incorporating not only competition but also beneficial interactions, which constitute the main interest of mixed cropping.

\section{Proofs of the subsection 2.1}

\subsection{Proof of proposition 1}
\label{proof_prop1}
\begin{proof}
  Let us set $\forall i\in \llbracket 1;N\rrbracket,~R_i = \displaystyle \log\left(\frac{S_i}{s_m}\right)$. We use the notation $ \theta = (x,y,S,\gamma)\in \Theta = \R^2\times (\R_+^*)^2$. Let us consider the functions
  \begin{equation}
    \begin{aligned}
      &g_r:(r_1, \theta_1,r_2, \theta_2)\in \R\times \Theta\times \R\times \Theta \mapsto \gamma_1\left(\log\left(\frac{S_1}{s_m}\right)\left(1-C_r(r_1,r_2,|\vec x_1-\vec x_2|)\right)-r_1\right)\\
      &C_r(r_1,r_2,|\vec x_1-\vec x_2|) = \frac{r_2}{2R_M\displaystyle\left(1+\frac{|\vec{x}_1-\vec{x}_2|^2}{\sigma_x^2}\right)}\left(1+\tanh\left(\frac{r_2-r_1}{\sigma_r}\right)\right)
    \end{aligned}
    \label{definition_gr}
  \end{equation}
  We set $ \theta_{1:N} = (x_i,y_i,S_i,\gamma_i)_{1\leq i\leq N}$ and the function
  \begin{equation}
    G_r(\circ, \theta_{1:N}) :  r_{1:N} = (r_i)_{1\leq i\leq N}\in \R^N\mapsto \left(\frac{1}{N-1}\sum_{j\neq i}g_r(r_i, \theta_i,r_j, \theta_j)\right)_{1\leq i\leq N}\in \R^N
  \end{equation}
  For all $i,j\in \llbracket 1;N\rrbracket$ and $ r_{1:N}\in \R^N$, we have
  \begin{equation}
    \begin{aligned}
      &|g_r(r_i,\theta_i,r_j,\theta_j)| \leq \gamma_{1:N}^{\max}\left(R_{1:N}^{\max}\left(1+|C(r_i,r_j,|\vec x_i- \vec x_j|)|\right) + |r_i|\right)\\
      &\text{where }\gamma_{1:N}^{\max} = \max_{1\leq i\leq N}\gamma_i,~R_{1:N}^{\max} = \max_{1\leq i\leq N} |R_i|\\
      &|C(r_i,r_j,| \vec x_i- \vec x_j|)|\leq \frac{|r_j|}{R_M}\\
      &\text{so }|g_r(r_i,\theta_i,r_j,\theta_j)| \leq K(\theta_{1:N}) (1+|r_i|+|r_j|)\\
      &\text{with } K(\theta_{1:N}) = \max\left(\gamma^{\max}_{1:N}R_{1:N}^{\max}, \gamma_{1:N}^{\max}\frac{R_{1:N}^{\max}}{R_M},\gamma_{1:N}^{\max}\right)
    \end{aligned}
  \end{equation}
  We consider the following norm over $\R^N$, defined by $\forall r_{1:N}\in \R^N,~|r_{1:N}| = \displaystyle \sum_{i=1}^N |r_i|$, known as the norm 1. We have for all $r_{1:N}\in \R^N$
  \begin{equation}
    \begin{aligned}
      |G_r(r_{1:N}, \theta_{1:N})|&\leq \frac{K(\theta_{1:N})}{N-1}\sum_{i=1}^N\sum_{j\neq i}(1+|r_i|+|r_j|) \\
      &\leq NK(\theta_{1:N})\left(1+|r_{1:N}|\right)
    \end{aligned}
  \end{equation}
  This inequality, along with the fact $G_r(\circ,\theta_{1:N})$ is a locally Lipschitz continuous map, proves that the differential system
  \begin{equation}
    \left\{\begin{array}{l}
     r_{1:N}(0) = \displaystyle \left(\log\left(\frac{s_i^0}{s_m}\right)\right)_{1\leq i\leq N}\\
    \displaystyle\forall t\in \R_+,~ \frac{\intd  r_{1:N}(t)}{\intd t} = G_r( r_{1:N}(t),\theta_{1:N})
    \end{array}\right.
  \end{equation}
  has an unique global solution defined over $\R_+$. Then the function $t\in \R_+,~s_{1:N} : t\in \R_+\mapsto \left(s_m e^{r_i(t)}\right)_{1\leq i\leq N}$ is the unique solution of system (\ref{system_schneider}). 
\end{proof}

\subsection{Proof of proposition 2}
\label{proof_prop2}

\begin{proof}
  We set $\forall t\in \R_+,~Z_N(t, Z_N^0) = (s_i(t, Z_N^0),x_i,y_i,S_i,\gamma_i)_{1\leq i\leq N}$. We consider the random intervall
  \begin{equation}
    T[{Z_N^0}] = \left\{t\in \R_+|\forall \tau \in [0;t]~ Z_N(\tau,{Z_N^0})\in \mathring{\domainD}^N\right\}
  \end{equation}
  The intervall $T[{Z_N^0}]$ is almost surely an intervall not reduced to singleton $\{0\}$, as $Z_N(0, Z_N^0)$ is in $\mathring{\domainD}^N$ almost surely and as $t\in \R_+\mapsto Z_N(t, Z_N^0)$ is a continuous mapping. $t^*( Z_N^0) = \sup(T[{Z_N^0}])$ is therefore positive random variable almost surely. Let $\Omega^* = \{\omega\in \Omega|t^*( Z_N^0(\omega)) > 0\}$ which is such that $\proba(\Omega^*) = 1$. Let $\omega \in \Omega^*$. Then for all $t\in [0;t^*( Z_N^0(\omega))[$, we have for all $i\in \llbracket 1;N\rrbracket$
  \begin{equation}
    \begin{aligned}
      &s_m < s_i(t, Z_N^0) < S_i(\omega)\\
      & \forall j\in \llbracket 1;N\rrbracket,~C(s_i(t, Z_N^0(\omega)),s_j(t, Z_N^0(\omega)),| x_i(\omega)- x_j(\omega)|)\in [0;1]\\
      &s_i(t, Z_N^0(\omega)) = s_i^0(\omega) + \int_0^t\gamma_i(\omega)s_i(\tau, Z_N^0(\omega))\left(\log\left(\frac{S_i(\omega)}{s_m}\right)\right.\\
      &\left.\times \left(1-\frac{1}{N-1}\sum_{j\neq i}C(s_i(\tau, Z_N^0(\omega)),s_j(\tau, Z_N^0(\omega)),| x_i(\omega)- x_j(\omega)|)\right)-\log\left(\frac{s_i(\tau, Z_N^0(\omega))}{s_m}\right)\right)\intd \tau\\
      &\text{so } s_i(t, Z_N^0(\omega))\leq s_i^0(\omega)+\gamma_i(\omega)\int_0^ts_i(\tau, Z_N^0(\omega))\log\left(\frac{S_i(\omega)}{s_i(\tau, Z_N^0(\omega))}\right)\intd \tau\\
      &\text{and } s_i(t, Z_N^0(\omega))\geq s_i^0(\omega) -\gamma_i(\omega)\int_0^t s_i(\tau, Z_N^0(\omega))\log\left(\frac{s_i(\tau, Z_N^0(\omega)}{s_m}\right)\intd \tau
    \end{aligned}
  \end{equation}
  According to Grönwall lemma, the latest inequalities lead to
  \begin{equation}
    \begin{aligned}
      &s_i(t, Z_N^0(\omega)) \leq S_i(\omega) \exp\left(-e^{-\gamma_i(\omega)t}\log\left(\frac{S_i(\omega)}{s_i^0(\omega)}\right)\right) < S_i(\omega)\\
      &\text{and }s_i(t, Z_N^0(\omega)) \geq s_m\exp\left(-e^{-\gamma_i(\omega)t}\log\left(\frac{s_m}{s_i^0(\omega)}\right)\right) > s_m
    \end{aligned}
    \label{inequality21}
  \end{equation}
  If $t^*( Z_N^0(\omega)) < +\infty$, we can use inequalities (\ref{inequality21}) to obtain that $Z_N(t^*( Z_N^0(\omega)), Z_N^0(\omega))\in \mathring{\domainD}^N$. By continuity and by the fact that $\mathring{\domainD}^N$ is a non-empty open set, we can find $\epsilon(\omega) > 0$ such that $Z_N(t^*( Z_N^0(\omega))+\epsilon(\omega), Z_N^0(\omega))\in \mathring{\domainD}^N$, which is in contradiction with the definition of $t^*( Z_N^0(\omega))$. So $\forall \omega \in \Omega^*,~ t^*( Z_N^0(\omega)) = +\infty$.
\end{proof}

\section{Proof of the subsection \ref{sub_empirical_measure} : proposition 3}
\label{proof_prop3}
\begin{proof}
  We only have to check that the trajectory $t\in \R_+\mapsto \mu[t,Z_N^0]\in \probaspace_1(\Z)$ is continuous for the metric $W_1$. Let $\phi\in \C_L(\Z)$ a Lipschitz continuous function such that $\mathrm{Lip}(\phi)\leq 1$ and let $t_1,t_2\in \R_+$.
  \begin{equation*}
    \begin{aligned}
      &\left|\int_\Z\phi(z)\mu[t_1,Z_N^0](\intd z) - \int_\Z \phi(z)\mu[t_2,Z_N^0](\intd z) \right| = \left|\frac{1}{N}\sum_{i=1}^N\left(\phi(z_i(t_1,Z_N^0)-\phi(z_i(t_2,Z_N^0)\right)\right|\\
      &\leq \frac{1}{N}\sum_{i=1}^N\left|z_i(t_1,Z_N^0)-z_i(t_2,Z_N^0)\right|\\
      &\text{so }W_1(\mu[t_1,Z_N^0],\mu[t_2,Z_N^0]) \leq \frac{1}{N}\sum_{i=1}^N\left|z_i(t_1,Z_N^0)-z_i(t_2,Z_N^0)\right|
    \end{aligned}
  \end{equation*}
  It follows that $t\in \R_+\mapsto \mu[t,Z_N^0]$ is continuous for the metric $W_1$, by continuity of the solution of the system (\ref{system_micro}). The other recquirement for $t\in \R_+\mapsto \mu[t,Z_N^0]$ to be a measure solution is given by equation (\ref{weak_transport_micro}).
\end{proof}

\section{Proofs of the subsection \ref{sub_macro}}

\subsection{Proof of theorem \ref{thm_flow_X}}
\label{proof_thm1}

\begin{proof}{(\textit{of theorem 1})}
  Let us start by proving the local existence of functions satisfying the characteristic flow equation. Let $\alpha > 0$. We introduce the following functional space $\Y_\alpha = \C^0([-\alpha;\alpha]\rightarrow \Y)$ endowed with the functional metric $\displaystyle f\in \Y_\alpha\mapsto \|f\|_{\Y_\alpha} = \sup_{t\in [-\alpha;\alpha]}\|f(t,.)\|_\Y$. $\Y_\alpha$ is a Banach space for this metric. Over the functional space $\Y_\alpha$, we define the following map
  \begin{equation*}
    \begin{aligned}
      &\forall f\in \Y_\alpha,~\forall t\in [-\alpha;\alpha],~\forall (X,\theta)\in \Z,\\
      &\Phi_\alpha(f,t,X,\theta) = X + \int_0^t\int_{\Z}g(f(s,X,\theta),\theta,f(s,X',\theta'),\theta')\mu_0(\intd X',\intd\theta')\intd s \in \X
    \end{aligned}
  \end{equation*}
  For all $f\in \Y_\alpha$, $\Phi_\alpha(f,\cdot,\cdot)\in \Y_\alpha$. Let $R > 1$ and $\Y_{\alpha,R} = \{f\in \Y_\alpha|\|f\|_{\Y_\alpha}\leq R\}$. We have for all $f\in \Y_{\alpha,R}$, for all $(X,\theta)\in \Z$ and for all $t\in [-\alpha;\alpha]$
  \begin{equation*}
    \begin{aligned}
      &|\Phi_\alpha[f](t,X,\theta)|\leq |X| + K_1\int_0^t\int_{\Z}(1+|f(s,X,\theta)|+|f(s,z')|)\mu_0(\intd z')\intd s\\
      &|\Phi_\alpha[f](t,X,\theta)|\leq |X| + K_1\alpha\left(1+\|f\|_{\Y_\alpha}\left(2+|X|+|\theta| + \int_\Z |z'|\mu_0(\intd z')\right)\right)
    \end{aligned}
  \end{equation*}
  we set $\displaystyle M^1_{\mu_0} = \int_\Z |z'|\mu_0(\intd z')$the first order moment, then
  \begin{equation}
    \|\Phi_\alpha(f,\cdot,\cdot)\|_{\Y_\alpha} \leq 1+K_1\alpha(1+(2+M^1_{\mu_0})R)
    \label{inequality0}
  \end{equation}
  So if we choose $\alpha$ such that $1+K_1\alpha(1+(2+M^1_{\mu_0})R)\leq R$, i.e.
  \begin{equation}
    \alpha \leq \frac{R-1}{K_1(1+(2+M^1_{\mu_0})R)}
    \label{inequality1}
  \end{equation}
  we have that for all $f\in \Y_{\alpha,R}$, $\Phi_\alpha(f,\cdot,\cdot)\in \Y_{\alpha,R}$.\\
  Let $f_1,f_2\in \Y_{\alpha,R}$. We have for all $z = (X,\theta)\in \Z$ and for all $t\in [-\alpha;\alpha]$
  \begin{equation*}
    \begin{aligned}
      &|\Phi_\alpha(f_1,t,z)-\Phi_\alpha(f_2,t,z)|\leq\\
      &\int_0^t\int_{\Z}\left|g(f_1(s,z),\theta,f_1(s,X',\theta'),\theta') - g(f_2(s,z),\theta,f_2(s,X',\theta'),\theta')\right|\mu_0(\intd X',\intd \theta')\intd s\\
      &\leq K_2\int_0^t\int_{\Z}(1+|f_1(s,z')|+|f_2(s,z')|)(|f_1(s,z)-f_2(s,z)|+|f_1(s,z')-f_2(s,z')|)\mu_0(\intd z')\intd s\\
      &\leq K_2\alpha\|f_1-f_2\|_{\Y_\alpha}\int_\Z(1+2R(1+|z'|))(2+|z|+|z'|)\mu_0(\intd z')\\
    \end{aligned}
  \end{equation*}
  \begin{equation*}
    \begin{aligned}
      &\text{we set }M_{\mu_0}^2 = \int_\Z |z'|^2\mu_0(\intd z')\text{ then }\\
      &\|\Phi_\alpha(f_1,\cdot,\cdot)-\Phi_\alpha(f_2,\cdot,\cdot)\|_{\Y_\alpha}\leq K_2\alpha(2+4R+(1+6R)M_{\mu_0}^1 + 2RM_{\mu_0}^2)\|f_1-f_2\|_{\Y_\alpha}
    \end{aligned}
  \end{equation*}
  If $\alpha$ is chosen such that $K_2\alpha(2+4R+(1+6R)M_{\mu_0}^1 + 2RM_{\mu_0}^2) < 1$, and such that it satisfies the inequality (\ref{inequality1}), i.e.
  $$\displaystyle\alpha <  \min\left( \frac{R-1}{K_1(1+(2+M^1_{\mu_0})R)},\frac{1}{K_2(2+4R+(1+6R)M_{\mu_0}^1 + 2RM_{\mu_0}^2)}\right)$$
  then $\Phi_\alpha$ is a contractive map over $\Y_{\alpha,R}$. According to fixed-point theorem, there exists an unique $f_{\alpha,R}\in \Y_{\alpha,R}$ such that $\Phi_\alpha(f_{\alpha,R},\cdot,\cdot) = f_{\alpha,R}$, i.e. for all $t\in [-\alpha;\alpha]$ and $z = (X,\theta)\in \Z$, we have
  \begin{equation*}
    f_{\alpha,R}(t,z) = X + \int_0^t \int_{\Z}g(f_{\alpha,R}(s,z),\theta,f_{\alpha,R}(s,X',\theta'),\theta')\mu_0(\intd X',\intd\theta')\intd s
  \end{equation*}
  Let us prove now that any function satisfying the equation on a sub-interval of $[-\alpha;\alpha]$ is the restriction of the previous function $f_{\alpha,R}$ to this sub-interval. Without loss of generality, we work on sub-intervals of type $[-\beta;\beta]$ with $\beta \leq \alpha$.\\
  Let $f_\beta : [-\beta;\beta]\times \Z\rightarrow \X$ be such that
  \begin{equation*}
    \begin{aligned}
      &\forall t\in [-\beta;\beta],~\forall (X,\theta)\in \Z,~\int_\Z |g(f_\beta(t,X,\theta),\theta,f_\beta(t,X',\theta'),\theta')|\mu_0(\intd X',\intd \theta') <+\infty\\
      &\text{and }f_\beta(t,X,\theta) = X + \int_0^t\int_{\Z}g(f_\beta(s,X,\theta),\theta,f_\beta(s,X',\theta'),\theta')\mu_0(\intd X',\intd \theta')\intd s
    \end{aligned}
  \end{equation*}
  We can then distinguish two cases :
  \begin{enumerate}
  \item Either $\displaystyle \sup_{-\beta\leq t\leq \beta}\sup_{z\in \Z}\frac{|f_\beta(t,z)|}{1+|z|} \leq R$. Then, by following the same reasoning as previously, $f_\beta$ is the unique fixed point of the map $\Phi_\beta$ over the set $\Y_{\beta,R}$. Since the restriction of $f_{\alpha,R}$ to the interval $[-\beta;\beta]$ is also a fixed point of $\Phi_\beta$, then we have that $(f_{\alpha,R})_{[-\beta;\beta]} = f_\beta$.
  \item Either $\displaystyle \sup_{-\beta\leq t\leq \beta}\sup_{z\in \Z}\frac{|f_\beta(t,z)|}{1+|z|} > R$. Let us introduce the following time $\beta_R = \sup\{\delta\in [0;\beta]|\forall t\in [-\delta;\delta],~\|f_\beta(t,.)\|_\Y\leq R\}$. Then $\beta_R > 0$ necessarily, since $\|f_\beta(0,.)\|_\Y = 1 < R$. For all $\delta\in [0;\beta_R[$, we have, by deriving the same inequalities as in (\ref{inequality0}),
    \begin{equation}
      \|f_\beta(\delta,.)\|_\Y\leq 1+K_1\beta_R(1+(2+M^1_{\mu_0})R)\leq R
      \label{inequality3}
    \end{equation}
    By continuity, the previous inequality is also valid for $\delta = \beta_R$. Since $\displaystyle \sup_{-\beta\leq t\leq \beta}\sup_{z\in \Z}\frac{|f_\beta(t,z)|}{1+|z|} > R$, we have that $\beta_R < \beta$. By reinjecting this inequality in (\ref{inequality3}), we have in fact in that
    \begin{equation*}
      \max\left(\|f_\beta(-\beta_R,.)\|_\Y,\|f_\beta(\beta_R,.)\|_\Y\right)\leq 1+K_1\beta_R(1+(2+M^1_{\mu_0})R)< R
    \end{equation*}
    which is in contradiction with the definition of $\beta_R$. So the current case 2 is absurd.
  \end{enumerate}
  We can extend the following reasoning to any interval of $\R$ containing 0. We define the following set of tuples

  \begin{equation}
    \begin{aligned}
      &\setS_{0,\mu_0} = \left\{(J,f_J)|J\text{ is an interval of }\R\text{ containing }0,~f_J\in C^0(J\rightarrow \Y)\text{ such that }\right.\\
      &\left.\forall t\in J,~\forall (X,\theta)\in \Z,~f_J(t,X,\theta) =  X + \int_0^t\int_{\Z}g(f_J(s,X,\theta),\theta,f_J(s,X',\theta'),\theta')\mu_0(\intd X',\intd \theta')\intd s\right\}
    \end{aligned}
  \end{equation}
  This set is non-empty as it contains at least $([-\alpha,\alpha],f_{\alpha,R})$ and all its restriction to sub-intervals. $\setS_{0,\mu_0}$ is partially ordered by the following relationship
  \begin{equation*}
    \forall (J_1,f_{J_1}), (J_2,f_{J_2})\in \setS_{0,\mu_0},~(J_1,f_{J_1})\prec (J_2,f_{J_2})\Leftrightarrow J_1\subsetneq J_2
  \end{equation*}
  Let us consider the set $\bar{\setS}_{0,\mu_0}$ of maximal elements of $\setS_{0,\mu_0}$, i.e.
  \begin{equation*}
    \bar{\setS}_{0,\mu_0} = \{(J,f_J)\in \setS_{0,\mu_0}|\nexists (J',f_{J'})\in \setS_{0,\mu_0},~(J,f_J)\prec (J',f_{J'})\}
  \end{equation*}
  We prove now that the set of maximal elements $\bar{\setS}_{0,\mu_0}$ is reduced to a singleton.\\
  Let $(J_1,f_{J_1}),(J_2,f_{J_2})$ be two maximal elements of $\bar{\setS}_{0,\mu_0}$. We consider $J = J_1\cap J_2$ and $T_+ = \{t\in J|t\geq 0,\forall s\in [0;t],~\forall z\in \Z,~f_{J_1}(s,z) = f_{J_2}(s,z)\}$. Let us assume by contradiction that $T_+ \neq J\cap\R_+$. If $t^* = \sup(T_+)$, we can exclude two cases :
  \begin{enumerate}
  \item If $t^* = +\infty$, then $T_+ = \R_+$, so $\R_+\subset J\cap\R_+$, leading to $J = \R_+ = T_+$, which is a contradiction. So $t^*$ must be finite, $t^*< +\infty$.
  \item If $t^*\in \partial J$, i.e. the boundary of interval $J$, then $t^* = \sup(J)$, and therefore $T_+\cap \R_+ = J\cap \R_+$, which is a contradiction. So under our assumptions, $t^*$ must be in the interior of the interval $J$.
  \end{enumerate}
  For all $t\in [0;t^*)$, we have $\|f_{J_1}(t,.)-f_{J_2}(t,.)\|_{\Y} = 0$, so by continuity $\|f_{J_1}(t^*,.)-f_{J_2}(t^*,.)\|_{\Y} = 0$. Let $\delta > 0$ such that $t^* + \delta \in J$ and such that $\forall t\in [t^*;t^*+\delta]$, $\max(\|f_{J_1}(t,.)\|_\Y,\|f_{J_2}(t,.)\|_\Y)\leq R^* = \|f_{J_1}(t^*,.)\|_\Y + 1$.\\
    Let $z = (X,\theta)$ and $t\in [t^*;t^*+\delta]$
    \begin{equation*}
      \begin{aligned}
        &f_{J_1}(t,z) - f_{J_2}(t,z) = \\
        &\int_{t^*}^t\int_{\Z}\left(g(f_{J_1}(s,z),\theta,f_{J_1}(s,X',\theta'),\theta')-g(f_{J_2}(s,z),\theta,f_{J_2}(s,X',\theta'),\theta')\right)\mu_0(\intd X',\intd \theta')\intd s\\
        &\|f_{J_1}(t,.) - f_{J_2}(t,.)\|_\Y\leq K_2(2+4R^*+(1+6R^*)M_{\mu_0}^1 + 2R^*M_{\mu_0}^2)\int_{t^*}^t\|f_{J_1}(s,.)-f_{J_2}(s,.)\|_\Y\intd s
      \end{aligned}
    \end{equation*}
    The last inequality implies that for all $t\in [t^*;t^*+\delta]$, $\|f_{J_1}(t,.)-f_{J_2}(t,.)\|_\Y = 0$ by Grönwall lemma, which is a contradiction with the definition of $t^*$. So we have necessarily that $T_+ = J\cap\R_+$. We can conduct the same reasoning to prove that $T_- = \{t\in J|t\leq 0,~\forall s \in [t;0],~ f_{J_1}(s,.) = f_{J_2}(s,.)\}$ is equal to $J\cap\R_-$. So the functions $f_{J_1}$ and $f_{J_2}$ coincide on $J = J_1\cap J_2$. If $J_1\cap J_2 \subsetneq J_1\cup J_2$, we could construct the following function
    \begin{equation*}
      \forall t \in J_1\cup J_2,~ f_{J_1\cup J_2}(t,.) = \left\{\begin{array}{l}
      f_{J_1}(t,.)\text{ if }t\in J_1\\
      f_{J_2}(t,.)\text{ if }t\in J_2
      \end{array}\right.
    \end{equation*}
    Then we would have that $(J_1\cup J_2,f_{J_1\cup J_2})\in \setS_{0,\mu_0}$ and that $(J_1,f_{J_1})\prec (J_1\cup J_2,f_{J_1\cup J_2})$ and $(J_2,f_{J_2})\prec (J_1\cup J_2,f_{J_1\cup J_2})$, which would be in contradiction with the maximality of $(J_1,f_{J_1}),(J_2,f_{J_2})$. So $J_1 = J_2$ and $f_{J_1} = f_{J_2}$, and $\bar{\setS}_{0,\mu_0}$ is reduced to a singleton.\\
    Let us prove now that the unique maximal element $(J,f_J)$ is in fact defined over $\R_+$, i.e. $\R_+\subset J$. We consider $t^* = \sup(J)$. Let us assume by contradiction that $t^*<+\infty$. Then we have necessarily that $t^*\notin J$. Otherwise, we could apply the same reasoning as for the local existence in the beginning of the proof, to the initial time $t^*$ and to the initial distribution $\mu_{t^*}$ the probability distribution of $(f_J(t^*,{z_0}),{\theta_0})$ where ${z_0} = ({X_0},{\theta_0})$ is a random variable of distribution $\mu_0$. So we would be able to extend the interval of definition $J$, which would be in contradiction with the maximality of $(J,f_J)$.\\
    Let $t\in [0;t^*[$ and $(X,\theta)\in \Z$, we have
    \begin{equation*}
      \begin{aligned}
        &f_J(t,X,\theta) = X + \int_0^t\int_{\Z} g(f_J(s,X,\theta),\theta,f_J(s,X',\theta'),\theta')\mu_0(\intd X',\intd \theta')\intd s\\
        &\|f_J(t,.)\|_\Y \leq 1+K_1\int_0^t(1+(2+M^1_{\mu_0})\|f_J(s,.)\|_\Y)\intd s\\
        &\text{so by Grönwall lemma }\\
        &\|f_J(t,.)\|_\Y \leq \frac{1}{2(1+M^1_{\mu_0})}((3+2M^1_{\mu_0})\exp(2K_1(1+M^1_{\mu_0})t)-1) = M(t)
      \end{aligned}
    \end{equation*}
    We use the last inequality to show that the derivative $\displaystyle t\in [0;t^*[\mapsto \frac{\partial f_J}{\partial t}(t,z)$ is bounded for all $z\in \Z$. Let $t\in [0;t^*[$, $(X,\theta)\in \Z$ 
    \begin{equation*}
      \begin{aligned}
        &\frac{\partial f_J}{\partial t}(t,X,\theta) = \int_{\Z} g(f_J(t,X,\theta),\theta,f_J(t,X',\theta'),\theta')\mu_0(\intd X',\intd \theta')\\
        &\left|\frac{\partial f_J}{\partial t}(t,X,\theta)\right|\leq K_1(1+M(t^*)(2+|X|+|\theta|+M^1_{\mu_0}))
      \end{aligned}
    \end{equation*}
    So  $\displaystyle \lim_{t\rightarrow t^*}\int_0^{t}\left|\frac{\partial f_J}{\partial t}(s,z)\right|\intd s$ is finite, and $\displaystyle \lim_{t\rightarrow t^*} f_J(t,z) = X + \lim_{t\rightarrow t^*}\int_0^t\frac{\partial f_J}{\partial t}(s,z)\intd s$ exists. Then we can define the following function
    \begin{equation*}
      \forall t \in J\cup\{t^*\},~\forall z \in \Z,~f_{J\cup\{t^*\}}(t,z) = \left\{\begin{array}{l}
      f_J(t,z)\text{ if }t\in J\\
      \displaystyle  X + \lim_{t\rightarrow t^*}\int_0^t\frac{\partial f_J}{\partial t}(s,z)\intd s\text{ if }t = t^*
      \end{array}\right.
    \end{equation*}
    We have then that $(J\cup\{t^*\},f_{J\cup\{t^*\}})\in \setS_{0,\mu_0}$ and that $(J,f_J)\prec (J\cup\{t^*\},f_{J\cup\{t^*\}})$, which is in contradiction with the maximality of $(J,f_J)$. So $t^* = +\infty$ and the maximal element is defined over $\R_+$. 
\end{proof}

\subsection{proof of corollary 1}
\label{proof_thm2}

\begin{lemma}
  Let $\mu_0\in \probaspace_2(\Z)$, $g$ satisfying assumptions (A1) and (A2), $z^0 = (X^0,\theta)$ a random variable of distribution $\mu_0$ and $X_\infty : \R_+\times \Z\rightarrow \X$ the flow solution of equation (\ref{equ_flow_X}). For all time $t\in \R_+$, we denote by $\mu[t]$ the distribution of the random variable $z^t = (X_\infty(t,z^0),\theta)$. Then $\mu[t]$ is a measure solution of the transport equation (\ref{weak_transport_macro}).
\end{lemma}
\begin{proof}
  Let $\phi \in \C_L(\Z)$ be a Lipschitz continuous function such that $\mathrm{Lip}(\phi)\leq 1$, and $t_1,t_2\in \R_+$.
  \begin{equation*}
    \begin{aligned}
      &\int_\Z \phi(z) \mu[t_1](\intd z) = \expect_{\mu[t_1]}(\phi(z^{t_1})) = \expect_{\mu_0}(\phi(X_\infty(t_1,z^0),\theta)) = \int_\Z \phi(X_\infty(t_1,X,\theta),\theta)\mu_0(\intd X,\intd \theta)\\
      &\left|\int_\Z \phi(z)\mu[t_1](\intd z) - \int_\Z \phi(z)\mu[t_2](\intd z)\right|\\
      &\leq \int_\Z |\phi(X_\infty(t_1,X,\theta),\theta)-\phi(X_\infty(t_2,X,\theta),\theta)|\mu_0(\intd X,\intd \theta)\\
      &\leq \int_\Z |X_\infty(t_1,X,\theta)-X_\infty(t_2,X,\theta)|\mu_0(\intd X,\intd \theta)\\
      &W_1(\mu[t_1],\mu[t_2])\leq \left(1+\int_\Z |z|\mu_0(\intd z)\right)\|X_\infty(t_1,.)-X_\infty(t_2,.)\|_\Y
    \end{aligned}
  \end{equation*}
  The continuity of $t\in\R_+\mapsto \mu[t]\in \probaspace_1(\Z)$ for the metric $W_1$ is therefore implied by the continuity of $t\in \R_+\mapsto X_\infty(t,.)\in \Y$ for the metric $\|.\|_\Y$.
  Let $\varphi$ be a test function. For the initial time $t= 0$, $\mu[0]$ is the distribution of $(X_\infty(0,z^0),\theta) = (X^0,\theta)$, which is $\mu_0$ by definition. So we have
  \begin{equation*}
    \int_\Z \varphi(0,z)\mu[0](\intd z) = \int_\Z \varphi(0,z)\mu_0(\intd z)
  \end{equation*}
   $t\in \R_+\mapsto \displaystyle \int_\Z \varphi(t,z)\mu[t](\intd z)$ is continuously differentiable and 
  \begin{equation*}
    \begin{aligned}
      &\frac{\intd}{\intd t}\int_\Z \varphi(t,z)\mu[t](\intd z) = \int_\Z\left(\frac{\partial \varphi}{\partial t}(t,X_\infty(t,X,\theta),\theta)\right.\\
      &\left.+ \frac{\partial \varphi}{\partial X}(t,X_\infty(t,X,\theta),\theta)^\T\int_\Z g(X_\infty(t,X,\theta),\theta,X_\infty(t,X,\theta'),\theta')\mu_0(\intd X',\intd \theta')\right)\mu_0(\intd X,\intd \theta)\\
      & = \int_\Z\left(\frac{\partial \varphi}{\partial t}(t,z) + \frac{\partial \varphi}{\partial X}(t,z)^\T\int_\Z g(z,z')\mu[t](\intd z')\right)\mu[t](\intd z)\\
    \end{aligned}
  \end{equation*}
  
\end{proof}

Before proving that this measure-solution is in fact the unique one, we need to establish some auxiliary results.

\begin{lemma}
  \label{lemma_flow_fix_measure}
  Let $\mu_0\in \probaspace_2(\Z)$ and $g$ satisfying assumptions (A1) and (A2). Let $t\in \R_+\mapsto \nu[t]\in \probaspace_1(\Z)$ be a trajectory in the space of probability measures continuous for the metric $W_1$. Then there exists an unique flow to the (ordinary) differential equation
  \begin{equation}
    \forall t_0\in \R_+,~\forall (X,\theta)\in \Z,~\left\{\begin{array}{l}
    X^\nu(t_0,t_0,X,\theta) = X\\
    \forall t\in \R_+,~\displaystyle \frac{\partial X^\nu}{\partial t}(t,t_0,X,\theta) = \G(\nu[t],X^\nu(t,t_0,X,\theta),\theta)
    \end{array}\right.
    \label{flow_nu}
  \end{equation}
  where $\G$ is non-local velocity field defined in equation (\ref{velocity_g}).
\end{lemma}
\begin{proof}
  Let $T >0$, $t\in [0;T]$, $z_1,z_2\in \Z$
  \begin{equation*}
    \begin{aligned}
      &|\G(\nu[t],z_1)-\G(\nu[t],z_2)|\leq \int_\Z|g(z_1,z)-g(z_2,z)|\nu[t](\intd z)\\
      &\leq K_2\left(1+2\int_\Z |z|\nu[t](\intd z)\right)|z_1-z_2|\\
      &\leq K_2\left(1+2\max_{0\leq t\leq T}\int_\Z |z|\nu[t](\intd z)\right)|z_1-z_2|
    \end{aligned}
  \end{equation*}
  It follows that for all $\theta \in \Theta$, the map $(t,X)\in \R_+\times \X\mapsto \G(\nu[t],X,\theta)$ is globally Lipschitz continuous over the intervall $[0;T]$, for any $T >0$. The proof is concluded by Cauchy-Lipschitz theorem. 
\end{proof}

The two following lemmas are classical results from dynamical systems theory and transport equations.

\begin{lemma}(\cite{Golse2013}, theorem 2.2.3)
  \label{lemma_regularity_flow}
  Let $a:(t,X)\in \R_+\times \X\mapsto a(t,x)\in \X$ such that $a\in C(\R_+\times \X\rightarrow \X)$ and $\displaystyle  (t,X)\mapsto \frac{\partial a}{\partial X}(t,X)$ is defined and continuous over $\R_+\times \X$. We assume that there exists $K > 0$ such that for all $t\in \R_+$  and for all $X\in \X$, $|a(t,X)|\leq K(1+|X|)$, and we consider the flow associated to the differential equation
  \begin{equation}
    \forall t_0\in \R_+~\left\{\begin{array}{l}
      \forall X\in \X,~X^a(t_0,t_0,X) = X\\
      \displaystyle \forall t\in [0;T],~\forall X\in \X,~\frac{\partial X^a}{\partial t}(t,t_0,X) = a(t,X^a(t,t_0,X))
      \end{array}\right.
      \label{eq_transport_a}
  \end{equation}
  Then the flow $X^a$ is continuously differentiable with respect to its three arguments, i.e. $X^a\in \C^1(\R_+\times \R_+\times \X\rightarrow \X)$.
\end{lemma}

\begin{lemma}(\cite{Golse2013}, theorem 2.2.4)
  \label{lemma_linear_transport}
  Let $\varphi_0\in C^1(\X\rightarrow\R)$ and $(t,X)\in \R_+\times \X\mapsto a(t,X)\in \X$ be such that $a\in C(\R_+\times \X\rightarrow \X)$ and $\displaystyle \frac{\partial a}{\partial X}\in C(\R_+\times \X\rightarrow \M_{d_\X\times d_\X}(\R))$. We assume that for some $T > 0$ there exists $K > 0$ such that for all $t\in [0;T]$, $|a(t,X)|\leq K(1+|X|)$. Then there exists an unique solution $\varphi\in C^1([0;T]\times \X\rightarrow \R)$ to the partial differential equation
  \begin{equation*}
    \forall X\in \X,~\left\{\begin{array}{l}
    \varphi(0,X) = \varphi_0(X)\\
    \forall t\in [0;T],~\displaystyle \frac{\partial \varphi}{\partial t}(t,X) +  a(t,X)^\T\frac{\partial \varphi}{\partial X}(t,X) = 0
    \end{array}\right.
  \end{equation*}
  The solution $\varphi$ has the following expression
  \begin{equation*}
    \forall t\in [0;T],~\forall X\in \X,~\varphi(t,X) = \varphi_0(X^a(0,t,X))
  \end{equation*}
  where $(t,t_0,X)\in [0;T]\times [0;T]\times \X\rightarrow X^a(t,t_0,X)$ is the flow of equation (\ref{eq_transport_a}).
\end{lemma}
\begin{lemma}
  Let $\mu_0\in \probaspace_2(\Z)$, g satisfying assumptions (A1), (A2) and (A3), and $(X^0,\theta)$ a random variable of distribution $\mu_0$. Then the unique measure-solution to the transport equation (\ref{weak_transport_macro}) is $t\in \R_+\mapsto \mu[t]\in \probaspace_2(\Z)$ where for all $t\in \R_+$, $\mu[t]$ is the probability distribution of ${z}^t = (X_\infty(t,X^0,\theta),\theta)$.
\end{lemma}

\begin{proof}
  Let $t\in \R_+\mapsto \nu[t]\in \probaspace_1(\Z)$ be a measure solution to equation (\ref{weak_transport_macro}). Let us consider the flow $X^\nu$ associated to the differential equation (\ref{flow_nu}).
  Thanks to assumption (A3), we have by Leibniz integral rule
  \begin{equation*}
    \forall t\in \R_+,~\forall (X,\theta)\in \Z,~\frac{\partial \G}{\partial X}(\nu[t],X,\theta) = \int_\Z \frac{\partial g}{\partial X}(X,\theta,X',\theta')\nu[t](\intd X',\intd \theta')
  \end{equation*}
  According to lemma \ref{lemma_regularity_flow}, for all $\theta\in \Theta$, the map $(t,t_0,X)\in\R_+\times \R_+\times \X\mapsto X^\nu(t,t_0,X,\theta)$ is continuously differentiable with respect to $t,t_0$ and $X$.
  
  Let $\varphi_0\in \C^1_0(\Z\rightarrow \R)$, i.e. continuously differentiable and such that $\displaystyle \lim_{|z|\rightarrow +\infty}|\varphi_0(z)| + \left|\frac{\partial\varphi_0(z)}{\partial z}\right| = 0$. We consider the linear transport equation of unknown $\varphi$
  \begin{equation*}
    \forall z \in \Z,~\left\{\begin{array}{l}
    \varphi(0,z) = \varphi_0(z)\\
    \forall t \in \R_+,~\displaystyle \frac{\partial\varphi}{\partial t}(t,z) +  \frac{\partial \varphi}{\partial z}(t,z)^\T \G(\nu[t],z) = 0
    \end{array}\right.
  \end{equation*}
  Then, using lemma \ref{lemma_linear_transport}, the unique solution of above equation is
  \begin{equation*}
    \forall t\in \R_+, ~\forall (X,\theta)\in \Z,~\varphi(t,X,\theta) = \varphi_0(X^\nu(0,t,X,\theta),\theta)
  \end{equation*}
  From previously, we have that $\varphi\in \C^1_0(\R_+\times \Z\rightarrow \R)$. As $t\mapsto \nu[t]$ is a measure solution, we can write for all time $t\in \R_+$
  \begin{equation*}
    \begin{aligned}
      &\int_\Z \varphi(t,z)\nu[t](\intd z) - \int_\Z\varphi(0,z)\mu_0(\intd z) =\\
      &\int_0^t\int_\Z\left(\frac{\partial\varphi}{\partial t}(s,z) +  \frac{\partial \varphi}{\partial z}(s,z)^\T \G(\nu[s],z)\right)\nu[s](\intd z)\intd s\\
      &\text{so }\int_\Z\varphi(t,z)\nu[t](\intd z) = \int_\Z\varphi_0(z)\mu_0(\intd z)
    \end{aligned}
  \end{equation*}
  If we introduce for all time $t\in \R_+$, a random variable $z^t = (X^t,\theta^t)$ of distribution $\nu[t]$, we can rewrite the above equation as $\expect_{\nu[t]}(\varphi(t,X^t,\theta^t)) = \expect_{\mu_0}(\varphi_0(X^0,\theta))$. Let us introduce the random variable $z^{-t} = (X^{-t},\theta^t) = (X^\nu(0,t,X^t,\theta^t),\theta^t)$ and $\nu_{-t}[t]$ its probability distribution. We have then
  \begin{equation*}
    \begin{aligned}
      &\expect_{\nu[t]}(\varphi(t,X^t,\theta^t)) = \expect_{\nu_{-t}[t]}(\varphi(t,X^\nu(t,0,X^{-t}),\theta^t))\\
      &\varphi(t,X^\nu(t,0,X^{-t},\theta^t),\theta^t) = \varphi_0(X^\nu(0,t,X^\nu(t,0,X^{-t},\theta^t),\theta^t),\theta^t) = \varphi_0(X^{-t},\theta^t)\\
      & \expect_{\nu_{-t}[t]}(\varphi_0(X^{-t},\theta^t)) = \expect_{\mu_0}(\varphi_0(X^0,\theta))
    \end{aligned}
  \end{equation*}
  Hence, as the last equality holds for any $\varphi_0$ verifying $\displaystyle \lim_{|z|\rightarrow +\infty}|\varphi_0(z)| + \left|\frac{\partial\varphi_0(z)}{\partial z}\right| = 0$, the distributions $\nu_{-t}[t]$ and $\mu_0$ are equal for all time $t\in \R_+$. $z^t$ has the same distribution as the random variable $(X^\nu(t,0,X^0,\theta),\theta)$ and therefore for all $t\in \R_+$ and for all $(X,\theta)\in \Z$, we have
  \begin{equation*}
    \begin{aligned}
      &\int_\Z g(X^\nu(t,0,X,\theta),\theta,X',\theta')\nu[t](\intd X',\intd \theta')\\
      &= \int_\Z g(X^\nu(t,0,X,\theta),\theta,X^\nu(t,0,X',\theta'),\theta')\mu_0(\intd X',\intd \theta') = \frac{\partial X^\nu}{\partial t}(t,0,X,\theta)
    \end{aligned}
  \end{equation*}
  By unicity of the characteristic flow, it follows that $\forall z\in \Z,~\forall t\in \R_+,~X^\nu(t,0,z) = X_\infty(t,0,z)$. 
\end{proof}

\section{Proofs of the subsection \ref{sub_dobrushin}}
\label{proof_sub_dobrushin}

\subsection{Proof of lemma \ref{lemma_empirical_flow}}

\label{proof_lemma_empirical_flow}

\begin{lemma}
  \label{lemma_empirical_flow}
  Let $\mu_0\in\probaspace_1(\Z)$ and $g:\Z^2\rightarrow \X$ a transition function satisfying assumptions (A1), (A2). For any initial configuration of the population $Z_N^0\in \Z^N$, with $N >1$, there exists an unique function $\hat X(Z_N^0,.,.):(t,z)\mapsto \hat X(Z_N^0,t,z)\in \X$ such that
  \begin{equation}
    \begin{aligned}
      &\forall z = (X,\theta) \in \Z,\left\{\begin{array}{l}
      \hat X(Z_N^0,0,X,\theta) = X\\
      \displaystyle \forall t\in \R_+,~\frac{\partial \hat X}{\partial t}(Z_N^0,t,X,\theta) = \G_N(\mu[t,Z_N^0],\hat X(Z_N^0,t,X,\theta),\theta)
      \end{array}\right.\\
      &\text{where }\forall z \in \Z,~ \G_N(\mu[t,Z_N^0],z) = \int_\Z g_N(z,z')\mu[t,Z_N^0](\intd z')\\
      &\forall (z,z')\in \Z^2,~g_N(z,z') = \frac{N}{N-1}g(z,z')-\frac{g(z,z)}{N-1}
    \end{aligned}
    \label{eq_empirical_flow}
  \end{equation}
\end{lemma}

Let $\theta\in \Theta$. We consider the velocity field $(t,X)\in \R_+\times \X\mapsto \G(\mu[t,Z_N^0],X,\theta)\in \X$, where $\G_N$ is the non local velocity field defined in equation (\ref{velocity_micro}). Let $t\in \R_+$ and $X_1,X_2\in \X$.
\begin{equation*}
  \begin{aligned}
    &|\G_N(\mu[t,Z_N^0],X_1,\theta) - \G_N(\mu[t,Z_N^0],X_2,\theta)|\leq\\
    &\frac{N}{N-1} \int_\Z |g(X_1,\theta,X',\theta') - g(X_2,\theta,X',\theta')|\mu[t,Z_N^0](\intd X',\intd \theta')+ \frac{|g(X_1,\theta,X_1,\theta)-g(X_2,\theta,X_2,\theta)|}{N-1}\\
    &\leq K_2\left(\frac{N}{N-1}(1+2M^1(t,Z_N^0)) + \frac{1+|X_1|+|X_2|}{N-1}\right)|X_1-X_2|\\
    &\text{ with }M^1(t,Z_N^0) = \int_\Z|z|\mu[t,Z_N^0](\intd z)
  \end{aligned}
\end{equation*}
So $(t,X)\in \R_+\times \X\mapsto \G_N(\mu[t,Z_N^0],X)$ is locally Lipschitz continuous with respect to the variable $X$. Besides, for any time $T >0$ and $X\in \X,t\in [0;T]$
\begin{equation*}
  |\G_N(\mu[t,Z_N^0],X,\theta)|\leq K_1\left(\frac{N}{N-1}\max_{0\leq t\leq T} M^1(t,Z_N^0) + \frac{N+2}{N-1}(1+|X|)\right)
\end{equation*}
According to Cauchy-Lipschitz theorem, there exists an unique flow for the differential equation of velocity $(t,X)\in \R_+\times \X\mapsto \G_N(\mu[t,Z_N^0],X,\theta)$.
\begin{equation*}
  \forall (X,\theta)\in \Z, ~\left\{\begin{array}{l}
  \hat{X}(Z_N^0,0,X,\theta) = X\\
  \displaystyle \forall t\in \R_+,~\frac{\partial \hat{X}}{\partial t}(Z_N^0,t,X,\theta) = \G_N(\mu[t,Z_N^0],\hat{X}(Z_N^0,t,X,\theta),\theta)
  \end{array}\right.
\end{equation*}
Let $t\in \R_+\mapsto (X_i(t,Z_N^0))_{1\leq i\leq N}$ be the trajectories solution of the system (\ref{system_micro}). For all $i\in \llbracket 1;N\rrbracket$, $X_i(0,Z_N^0) = \hat{X}(Z_N^0,0,X_i^0,\theta_i) = X_i^0$ and $\forall t\in \R_+,$ $\displaystyle \frac{\intd X_i}{\intd t}(t,Z_N^0) = \G_N(\mu[t,Z_N^0],X_i(t,Z_N^0),\theta_i)$. It follows by unicity that $\forall t\in \R_+,$ $\hat{X}(Z_N^0,t,X_i^0,\theta_i) = X_i(t,Z_N^0)$. We finally obtain that
\begin{equation*}
  \begin{aligned}
    &\forall (X,\theta)\in \Z, ~\frac{\partial \hat X}{\partial t}(Z_N^0,t,X,\theta) = \frac{1}{N}\sum_{i = 1}^N g_N(\hat X(Z_N^0,t,X,\theta),\theta,X_i(t,Z_N^0),\theta_i)\\
    &= \frac{1}{N}\sum_{i = 1}^N g_N(\hat X(Z_N^0,t,X,\theta),\theta, \hat X(Z_N^0,t,X_i^0,\theta_i),\theta_i)\\
    & = \int_\Z g_N(\hat X(Z_N^0,t,X,\theta),\theta,\hat X(Z_N^0,t,X',\theta'),\theta')\mu[0,Z_N^0](\intd X',\intd \theta')
  \end{aligned}
\end{equation*}

\subsection{proof of theorem \ref{thm_dobrushin}}
\label{proof_thm_dobrushin}
For the simplicity of notations, we introduce the the map $\hat{Z}(Z_N^0,.,.):(t,X,\theta)\in \R_+\times \Z\mapsto (\hat X(Z_N^0,t,X,\theta),\theta)\in \Z$ and the map $Z:(t,X,\theta)\in \Z\mapsto (X_\infty(t,X,\theta),\theta)\in \Z$. Let $\pi_0$ be a probability distribution in the set of couplings $\Pi(\mu[0,Z_N^0],\mu_0)$ and let $(z_1,z_2)$ be a random variable of distribution $\pi_0$. We consider for all time $t\in \R_+$, the distribution $\pi_t$ of the random variable $(\hat Z(Z_N^0,t,z_1),Z(t,z_2))$. Then it is straightforward that $\pi_t$ is in the set of couplings $\Pi(\mu[t,Z_N^0],\mu[t])$. We then have the following inequality.
\begin{equation*}
  \begin{aligned}
    & W_1\left(\mu[t,Z_N^0],\mu[t]\right)\leq \iint_{\Z^2}|z_1-z_2|\pi_t(\intd z_1, \intd z_2) \leq \iint_{\Z^2}|\hat Z(Z_N^0,t,z_1)-Z(t,z_2)|\pi_0(\intd z_1,\intd z_2)
  \end{aligned}
\end{equation*}
Let $z_1,z_2\in \Z$.
\begin{equation}
  \begin{aligned}
    & |\hat Z(Z_N^0,t,z_1)-Z(t,z_2)| \leq |z_1-z_2| + \int_0^t\left|\int_{\Z}g_N(\hat Z(Z_N^0,s,z_1),\hat Z(Z_N^0,s,z_1'))\mu[0,Z_N^0](\intd z_1')\right.\\
    & \left.- \int_{\Z}g(Z(s,z_2),Z(s,z_2'))\mu_0(\intd z_2')\right|\intd s\\
    & |\hat Z(Z_N^0,t,z_1)-Z(t,z_2)| \leq |z_1-z_2| + \int_0^t A_N(s,Z_N^0,z_1)\intd s + \int_0^tB_N^{\pi_0}(s,Z_N^0,z_1,z_2)\intd s\\
    &\text{where } \forall t\in \R_+,~ A_N(t,Z_N^0,z_1) = \\
    &\int_\Z |g_N(\hat Z(Z_N^0,t,z_1),\hat Z(Z_N^0,t,z_1')) - g(\hat Z(Z_N^0,t,z_1),\hat Z(Z_N^0,t,z_1'))|\mu[0,Z_N^0](\intd z_1')\\
    &\text{and }B_N^{\pi_0}(t,Z_N^0,z_1,z_2) = \iint_{\Z^2} |g(\hat Z(Z_N^0,t,z_1),\hat Z(Z_N^0,t,z_1'))-g(Z(t,z_2),Z(t,z_2'))|\pi_0(\intd z_1',\intd z_2')
  \end{aligned}
  \label{equ_A_B}
\end{equation}
Let us consider the term depending on $A_N$.
\begin{equation*}
  \begin{aligned}
    &\forall t\in \R_+,~ A_N(t,Z_N^0,z_1) = \\
    &\frac{1}{N-1}\int_\Z |g(\hat Z(Z_N^0,t,z_1),\hat Z(Z_N^0,t,z_1'))-g(\hat Z(Z_N^0,t,z_1),\hat Z(Z_N^0,t,z_1))|\mu[0,Z_N^0](\intd z_1')
  \end{aligned}
\end{equation*}
As $g$ satisfies assumptions (A2) and (A4), we have for all $z_1,z_1',z_2,z_2'\in \Z$
\begin{equation*}
  |g(z_1,z_1')-g(z_2,z_2')|\leq K_{24}(1+|z_1|+|z_1'|+|z_2|+|z_2'|)(|z_1-z_2|+|z_1'-z_2'|)
\end{equation*}
We use also the following notation for all time $t\in \R_+$ $\|\hat Z(Z_N^0,t,.)\| = \|\hat X(Z_N^0,t,.)\|_\Y$
\begin{equation*}
  \begin{aligned}
    & A_N(t,Z_N^0,z_1) \leq \\
    &\frac{K_{24}}{N-1}\int_\Z \left(1+\|\hat Z(Z_N^0,t,.)\|(4+3|z_1|+|z_1'|)\right)\|\hat Z(Z_N^0,t,.)\|(2+|z_1|+|z_1'|)\mu[0,Z_N^0](\intd z_1')
  \end{aligned}
\end{equation*}
We now look for an upper-bound of the function $t\in \R_+\mapsto \|\hat Z(Z_N^0,t,.)\|$.
\begin{equation*}
  \begin{aligned}
    &\forall z,z'\in \Z,~ |g_N(z,z')|\leq \frac{K_1}{N-1}(N+1 + (N+2)|z|+N|z'|)\leq \frac{N+2}{N-1}K_1(1+|z|+|z'|)\\
    &|g_N(z,z')|\leq K_N(1+|z|+|z'|)\text{ with }K_N = \frac{N+2}{N-1}K_1
  \end{aligned}
\end{equation*}
Now let us consider the empirical characteristic $\hat Z(Z_N^0,.,.)$.
\begin{equation*}
  \begin{aligned}
    &\forall t\in \R_+,~\forall z\in \Z,~|\hat Z(Z_N^0,t,z)|\leq |z| + \int_0^t\int_{\Z}|g_N(\hat Z(Z_N^0,s,z),\hat Z(Z_N^0,s,z'))|\mu[0,Z_N^0](\intd z')\\
    &\|\hat Z(Z_N^0,t,.)\| \leq 1+K_N\int_0^t(1+(2+M^1(0,Z_N^0))\|\hat Z(Z_N^0,s,.)\|)\intd s \\
    &\text{ with } M^1(0,Z_N^0) = \int_\Z |z'|\mu[0,Z_N^0](\intd z')\\
    &\text{so by Grönwall lemma }\|\hat Z(Z_N^0,t,.)\|\leq \frac{(3+2M^1(0,Z_N^0))\exp(2K_N(1+M^1(0,Z_N^0))t)-1}{2(1+M^1(0,Z_N^0))}
  \end{aligned}
\end{equation*}
We use the notation $\forall t\in \R_+,~ M_N(t,Z_N^0) = \displaystyle \frac{(3+2M^1(0,Z_N^0))\exp(2K_N(1+M^1(0,Z_N^0))t)-1}{2(1+M^1(0,Z_N^0))}$. Then we obtain that
\begin{equation*}
  \begin{aligned}
    & A_N(t,Z_N^0,z_1) \leq \\
    &\frac{K_{24}M_N(t,Z_N^0)}{N-1}\left(2+|z_1|+M^1(0,Z_N^0) + M_N(t,Z_N^0)\left((4+3|z_1|)(2+|z_1|) + (6+4|z_1|)M^1(0,Z_N^0)\right.\right.\\
    &\left.\left.+ M^2(0,Z_N^0)\right)\right)\\
    &\text{ with }M^2(0,Z_N^0) = \int_\Z |z'|^2\mu[0,Z_N^0](\intd z')
  \end{aligned}
\end{equation*}
Now let us consider the term depending on $B_N^{\pi_0}$.
\begin{equation*}
  \begin{aligned}
    &\forall t\in \R_+,~B_N^{\pi_0}(t,Z_N^0,z_1,z_2)\leq\\
    &K_{24}\iint_{\Z^2}\left(1+\|\hat Z(Z_N^0,t,.)\|(2+|z_1|+|z_1'|) + \|Z(t,.)\|(2+|z_2|+|z_2'|)\right)\\
    &\times \left(|\hat Z(Z_N^0,t,z_1)-Z(t,z_2)|+|\hat Z(Z_N^0,t,z_1')-Z(t,z_2')|\right)\pi_0(\intd z_1',\intd z_2')
  \end{aligned}
\end{equation*}
The same reasonning as previously can be applied here to the function $t\in\R_+\mapsto \|Z(t,.)\|$ to show that $\forall t\in \R_+,~ \|Z(t,.)\|\leq \displaystyle \frac{(3+2M^1_{\mu_0})\exp(2K_1(1+M^1_{\mu_0})t)-1}{2(1+M^1_{\mu_0})} = M(t)$, which leads to
\begin{equation*}
  \begin{aligned}
    &B_N^{\pi_0}(t,Z_N^0,z_1,z_2)\leq \\
    &K_{24}(1+M_N(t,Z_N^0)(2+R_0+|z_1|) + M(t)(2+R_0+|z_2|))\left(|\hat Z(Z_N^0,t,z_1)-Z(t,z_2)|\right.\\
    &\left.+\iint_{\Z^2}|\hat Z(Z_N^0,t,z'_1)-Z(t,z'_2)|\pi_0(\intd z_1',\intd z_2')\right)
  \end{aligned}
\end{equation*}
We use the previous inequalities to find an upper-bound of the quantity $D^{\pi_0}_N(t) = \displaystyle \iint_{\Z^2}|\hat Z(Z_N^0,t,z_1)-Z(t,z_2)|\pi_0(\intd z_1,\intd z_2)$.
\begin{equation*}
  \begin{aligned}
    & D^{\pi_0}_N(t) \leq \iint_{\Z^2}|z_1-z_2|\pi_0(\intd z_1,\intd z_2) + \int_0^t\int_\Z A(s,Z_N^0,z_1)\mu[0,Z_N^0](\intd z_1)\intd s +\\
    &\int_0^t\iint_{\Z^2}B_N^{\pi_0}(s,Z_N^0,z_1,z_2)\pi_0(\intd z_1,\intd z_2)\\
    & \forall t\in \R_+,\\
    &\int_\Z A_N(t,Z_N^0,z_1)\mu[0,Z_N^0](\intd z_1) \leq \\
    &\frac{2K_{24}M_N(t,Z_N^0)}{N-1}\left(1+M^1(0,Z_N^0) + 4M_N(t,Z_N^0)\left(1+2M^1(0,Z_N^0) + (M^1(0,Z_N^0))^2 + M^2(0,Z_N^0)^2\right)\right)\\
    &\text{we set }\\
    &E_N(t,Z_N^0) = 2K_{24}M_N(t,Z_N^0)\left(1+M^1(0,Z_N^0) + 4M_N(t,Z_N^0)\left(1+2M^1(0,Z_N^0) + (M^1(0,Z_N^0))^2 \right.\right.\\
    &\left.\left.+ M^2(0,Z_N^0)^2\right)\right)\\
    &\forall t\in \R_+,~ \iint_{\Z^2}B_N^{\pi_0}(t,z_1,z_2)\pi_0(\intd z_1,\intd z_2)\leq f_N(t,Z_N^0)\iint_{\Z^2}|\hat Z(Z_N^0,t,z_1)-Z(t,z_2)|\pi_0(\intd z_1,\intd z_2)\\
    &\text{where }f_N(t,Z_N^0) = 2K_{24}(1+M_N(t,Z_N^0)(2+R_0+M^1(0,Z_N^0)) + M(t)(2+R_0+M^1_{\mu_0}))\\
  \end{aligned}
\end{equation*}
\begin{equation*}
  \begin{aligned}
    &\text{so }D^{\pi_0}_N(t) \leq D^{\pi_0}_N(0) + \frac{1}{N-1}\int_0^tE_N(s,Z_N^0)\intd s + \int_0^tf_N(s,Z_N^0)D^{\pi_0}_N(s)\intd s\\
    & \text{by Grönwall lemma }D^{\pi_0}_N(t) \leq e^{F_N(t,Z_N^0)}\left(D^{\pi_0}_N(0) + \frac{1}{N-1}\int_0^tE_N(s,Z_N^0)e^{-F_N(s,Z_N^0)}\intd s\right)\\
    &\text{where } F_N(t,Z_N^0) = \int_0^tf_N(s,Z_N^0)\intd s
  \end{aligned}
\end{equation*}
By taking the infimum over $\Pi(\mu[0,Z_N^0],\mu_0)$, we obtain the inequality
\begin{equation*}
  W_1(\mu[t,Z_N^0],\mu[t]) \leq e^{F_N(t,Z_N^0)}\left(W_1(\mu[0,Z_N^0],\mu_0) + \frac{1}{N-1}\int_0^tE_N(s,Z_N^0)e^{-F_N(s,Z_N^0)}\intd s\right)
\end{equation*}
Let us study the convergence of the upper bound when the sequence of random variables $(Z_N^0)_{N > 1}$ is evaluated at $\omega\in \Omega^*$ such that $\lim_{N\rightarrow +\infty}W_1(\mu[0,Z_N^0(\omega)],\mu_0) = 0$. The last convergence implies in particular that $M^1(0,Z_N^0(\omega))\conv{N}{+\infty}{}M^1_{\mu_0}$. As the distribution $\mu_0$ have a compact support of diameter upper bounded by $2R_0$, we can write
\begin{equation*}
  \begin{aligned}
    &\forall \pi_0\in \Pi(\mu[0,Z_N^0],\mu_0),~\iint_{\Z^2}|z_1-z_2|^2\pi_0(\intd z_1,\intd z_2) \leq 2R_0\iint_{\Z^2}|z_1-z_2|\pi_0(\intd z_1,\intd z_2)\\
    & W_2(\mu[0,Z_N^0],\mu_0)\leq 2R_0W_1(\mu[0,Z_N^0],\mu_0)
  \end{aligned}
\end{equation*}
So $W_2(\mu[0,Z_N^0(\omega)],\mu_0) \conv{N}{+\infty}{} 0$, and therefore $M^2(0,Z_N^0(\omega))\conv{N}{+\infty}{} M^2_{\mu_0}$. It follows that $\forall t\in \R_+$
\begin{equation*}
  \begin{aligned}
    &E_N(t,Z_N^0(\omega)) \conv{N}{+\infty}{} E_{\mu_0}(t) = 2K_{24}M(t)\left(1+M^1_{\mu_0} + 4M(t)\left(1+2M^1_{\mu_0} + (M^1_{\mu_0})^2 + M^2_{\mu_0}\right)\right)\\
    &f_N(t,Z_N^0(\omega)) \conv{N}{+\infty}{} f_{\mu_0}(t) = 2K_{24}M(t) \left(1+2M(t)(2+R_0+M^1_{\mu_0})\right)\\
    &F_N(t,Z_N^0(\omega)) \conv{N}{+\infty}{} \int_0^t f_{\mu_0}(\tau)\intd\tau
  \end{aligned}
\end{equation*}
Then we obtain by dominated convergence that $W_1(\mu[t,Z_N^0(\omega)],\mu[t]) \conv{N}{+\infty}{} 0$.

\subsection{Proof of corollary \ref{corollary_dobrushin}}

Let us start by establishing an estimation of the quantity $|X_1(t,Z_N^0)-X_\infty(t,z_1^0)|$ for any time $t$ and for any $Z_N^0 = (z_1^0,...,z_N^0)\in \Z^N$.
\begin{equation*}
  \begin{aligned}
    &|X_1(t,Z_N^0)-X_\infty(t,z_1^0)|\leq \int_0^t\left|\G_N(s,z_1(s,Z_N^0))-\G(s,Z(s,z_1^0))\right|\intd s\\
    &\leq \int_0^tA_N(s,Z_N^0,z_1^0)\intd s + \int_0^tB_N^{\pi_0}(s,Z_N^0,z_1^0,z_1^0)\intd s
  \end{aligned}
\end{equation*}
where the functions $A_N$ and $B^{\pi_0}_N$ are defined in equation (\ref{equ_A_B}) with $\pi_0$ any coupling of $\Pi(\mu[0,Z_N^0],\mu_0)$. In the proof of theorem \ref{thm_dobrushin} (cf appendix section \ref{proof_thm_dobrushin}), we have proved the following inequalities
\begin{equation*}
  \begin{aligned}
    &A_N(t,Z_N^0,z_1^0)\leq \frac{e_N(t,Z_N^0)}{N-1} \text{ with }e_N(t,Z_N^0) = K_{24}M_N(t,Z_N^0)(2+|z_1^0|+M^1(0,Z_N^0)\\
    &+ M_N(t,Z_N^0)((4+3|z_1^0|)(2+|z_1^0|)+M^1(0,Z_N^0)(6+4|z_1^0|) + M^2(0,Z_N^0)))\\
    &B_N^{\pi_0}(t,Z_N^0,z_1^0,z_1^0) \leq K_{24}(1+(M_N(t,Z_N^0)+M(t))(2+R_0+|z_1^0|))\left(|z_1(t,Z_N^0)-Z(t,z_1^0)|\phantom{\iint}\right.\\
    &\left. + \iint_{\Z^2}|z_1'-z_2'|\pi_t(\intd z_1',\intd z_2')\right)\\
    &\text{we set }h_N(t,Z_N^0) = K_{24}(1+(M_N(t,Z_N^0)+M(t))(2+R_0+|z_1^0|))\\
  \end{aligned}
\end{equation*}
We use also as in the previous proof the notation $\displaystyle D^{\pi_0}_N(t,Z_N^0)= \iint_{\Z^2}|z_1'-z_2'|\pi_t(\intd z_1',\intd z_2')$. The argument of Dobrushin leads to the following inequality
\begin{equation*}
  D^{\pi_0}_N(t,Z_N^0) \leq e^{F_N(t,Z_N)}\left(D^{\pi_0}_N(0,Z_N^0) + \frac{1}{N-1}\int_0^tE_N(s,Z_N^0)e^{-F_N(s,Z_N^0)}\intd s\right)
\end{equation*}
By gathering the previous inequalities, we obtain finally
\begin{equation*}
  \begin{aligned}
    &|X_1(t,Z_N^0)-X_\infty(t,z_1^0)|\leq \frac{1}{N-1}\int_0^te_N(s,Z_N^0)\intd s + \int_0^th_N(s,Z_N^0)e^{F_N(s,Z_N)}\left(D^{\pi_0}_N(0,Z_N^0)\phantom{\int}\right.\\
    &\left.+ \frac{1}{N-1}\int_0^sE_N(\tau,Z_N^0)e^{-F_N(\tau,Z_N^0)}\intd \tau\right)\intd s + \int_0^t h_N(s,Z_N^0)|X_1(s,Z_N^0)-X_\infty(s,z_1^0)|\intd s
  \end{aligned}
\end{equation*}
As this inequality holds for any $\pi_0\in \Pi(\mu[0,Z_N^0],\mu_0)$, we can take $\pi_0$ equal to the optimal plan, so that $D^{\pi_0}_N(t,Z_N^0) = W_1(\mu[0,Z_N^0],\mu_0)$. By setting $k_N(t,Z_N^0) = \displaystyle \frac{e_N(t,Z_N^0)}{N-1} + h_N(t,Z_N^0)e^{F_N(t,Z_N^0)}\left(W_1(\mu[0,Z_N^0],\mu_0) + \frac{1}{N-1}\int_0^tE_N(s,Z_N^0)e^{-F_N(s,Z_N^0)}\intd s\right)$, we obtain by Grönwall lemma
\begin{equation*}
  |X_1(t,Z_N^0)-X_\infty(t,z_1^0)|\leq \int_0^tk_N(s,Z_N^0)\exp\left(\int_0^s(h_N(s,Z_N^0)-h_N(\tau,Z_N^0))\intd \tau\right)\intd s
\end{equation*}
It is clear that for all time $t\in \R_+$ and for all $\omega\in \Omega^*$, we have that $k_N(t,Z_N^0(\omega))\conv{N}{\infty}{}0$ and $h_N(t,Z_N^0(\omega))\conv{N}{\infty}{}K_{24}(1+2M(t)(2+R_0+|z_1^0(\omega)|))$. By an argument of dominated convergence, we can obtain that
\begin{equation*}
  \forall t\in \R_+,~\forall \omega\in \Omega^*,~X_1(t,Z_N^0(\omega))\conv{N}{\infty}{}X_\infty(t,z_1^0(\omega))
\end{equation*}

%\begin{acknowledgements}
%\end{acknowledgements}

\bibliographystyle{plainnat}      % basic style, author-year citations
\bibliography{biblio_springer}   % name your BibTeX data base

\end{document}